\documentclass[11pt]{article}
\usepackage{amstext, amsmath,latexsym,amsbsy,amssymb,amsmath}

\usepackage{enumitem}

\newtheorem{proposition}{Proposition}[section]
\newtheorem{remark}{Remark}[section]
\newenvironment{proof}{{\bf Proof\ }}{\QED\\}
\newtheorem{lemma}{Lemma}[section]

\numberwithin{equation}{section}
\newtheorem{theorem}{Theorem}[section]

\newcommand{\QED}{\hspace*{\fill}\rule{2.5mm}{2.5mm}}

\usepackage{amssymb}\usepackage{graphicx}
\usepackage{tikz}

\usepackage{color}
\newcommand\qed{\hfill$\sqcap\kern-7.5pt\hbox{$\sqcup$}$}

\newcommand{\RR}{\mathbb{R}}

\newcommand{\beqn}{\begin{equation}}
\newcommand{\eeqn}{\end{equation}}
\newcommand{\bear}{\begin{eqnarray}}
\newcommand{\eear}{\end{eqnarray}}
\newcommand{\bean}{\begin{eqnarray*}}
\newcommand{\eean}{\end{eqnarray*}}

\newcommand{\cE}{\omega}

\allowdisplaybreaks

\begin{document}
\title{On the wave turbulence theory  for stratified  flows in the ocean }

\author{Irene M. Gamba\footnotemark[1] \and Leslie M. Smith\footnotemark[2] \and Minh-Binh Tran\footnotemark[3]
}

\footnotetext[1]{Department of Mathematics and Institute for Computational Engineering and Sciences, The University of Texas at Austin, TX 78712, USA.\\
Email: gamba@math.utexas.edu
}

\footnotetext[2]{Department of Mathematics and Department of Engineering Physics, University of Wisconsin-Madison, Madison, WI 53706, USA. \\Email: lsmith@math.wisc.edu
}

\footnotetext[3]{Department of Mathematics, Southern Methodist University, Dallas, Texas, TX 75275, USA. \\Email: minhbinht@mail.smu.edu
}

\maketitle
\begin{abstract} 
After the pioneering work of Garrett and Munk, the statistics of oceanic internal gravity waves has become a central subject of research in oceanography.
The time evolution of the spectral energy of internal waves in the ocean can be described by a near-resonance wave turbulence equation, of quantum Boltzmann type. 
In this work, we provide the first rigorous mathematical study for the equation by showing the global existence and uniqueness of strong solutions.
 \end{abstract}




\section{Introduction}
The study of wave turbulence has obtained spectacular success in the understanding of spectral energy transfer processes in plasmas, oceans, and planetary atmospheres. Wave-wave interactions in continuously stratified fluids have been a fascinating subject of intensive research in the last few decades. In particular, the observation of a nearly universal internal-wave energy spectrum in the ocean, first described by Garrett and Munk (cf. \cite{garrett1975space,garrett1979internal,cairns1976internal}), plays a very important role in understanding  such wave-wave interactions. The existence of a universal spectrum is generally perceived to be  the result of nonlinear interactions of waves with different wavenumbers. As the nonlinearity of the underlying primitive equations is quadratic, waves interact in triads (cf. \cite{waleffe1992nature}). 
Furthermore, since  the linear internal wave dispersion relation can satisfy a three-wave resonance condition, resonant triads are expected to dominate the dynamics for weak nonlinearity  (cf. \cite{mccomasbretherton77}).

Resonant wave interactions can be characterized by Zakharov kinetic equations (cf. \cite{zakharov2012kolmogorov,Nazarenko:2011:WT,majda1997one,cai1999spectral,zakharov1967weak,zakharov1968stability}).  The equations   describe, under the assumption of weak nonlinearity, the 
spectral energy transfer on the resonant manifold, which is a set of wave vectors $k$, $k_1$, $k_2$ satisfying
\begin{equation}
\label{ResonantManifold}
k=k_1+k_2,\ \ \ \ \ \ \omega_k=\omega_{k_1}+\omega_{k_2},
\end{equation}
where the frequency $\omega$ is given by 
the dispersion relation between the wave frequency $\omega$ and the  wavenumber $k$.
However, it is known that exact resonances defined by $\omega_k=\omega_{k_1}+\omega_{k_2}$ do not capture some important physical effects, such as energy transfer to
non-propagating wave modes with zero frequency, corresponding to generation of anisotropic coherent structures \cite{babin96,babin2000,bartello95,embidmajda96,embidmajda98,greenspan69,lelongriley91,longuet67,LukkarinenSpohn:2011:WNS,majdaembid98,phillips68,waleffe93,warn86}, see also \cite{EscobedoVelazquez:2015:OTT,Merino:Thesis:2015} for analytical arguments on reduced isotropic models. Some authors have included more physics by  allowing near-resonant interactions (cf. 
\cite{connaughton2001discreteness,lee2007formation,lvov2012resonant,l1997statistical,lvov2004hamiltonian,lvov2004noisy,lvov2010oceanic,newell69,smith2005near,remmel2009new,remmel2010nonlinear}), 
defined as
\begin{equation}
\label{NearResonantManifold}
k=k_1+k_2,\ \ \ \ |\omega_k-\omega_{k_1}-\omega_{k_2}|<\theta(f,k), 
\end{equation}
where $\theta$ accounts for broadening of the resonant surfaces and  depends on the wave density $f$ and the wave number $k$. When near resonances are included in the dynamics, numerical studies have demonstrated the formation of the
anisotropic, non-propagating wave modes in dispersive wave systems relevant to geophysical flows (cf. \cite{chekhlov96,huang2000,lee2007formation,remmel2014nonlinear,smith2001,smith2005near,smith1999transfer,smith2002generation}).

We consider in this paper the following near-resonance turbulence kinetic equation for internal wave interactions in the open ocean (cf. \cite{connaughton2001discreteness,l1997statistical,lvov2004hamiltonian,lvov2004noisy,lvov2012resonant}), 
\begin{equation}\label{WeakTurbulenceInitial}
\begin{aligned}
\partial_tf(t,k) + \mu_k f(t,k) \ =& \ Q[f](t,k), \ \ \ f(0,k)=f_0(k),
\end{aligned}
\end{equation}
in which  $f(t,k)$ is the nonnegative wave density  at  wavenumber $k \in \RR^d$, $d \ge 2$. As proposed by  Zakharov in  \cite{zakharov1967weak} water wave models  must include the term $\mu_kf=2\nu|k|^2 f$ for viscous damping effects, with $\nu$  the viscosity coefficient.

This model equation consist in a  kinetic three-wave interaction modelled by an interaction (or collision) operator given by the non-local form  
\begin{equation}\label{def-Qf}
Q[f](k) \ = \ \iint_{\mathbb{R}^{2d}} \Big[ R_{k,k_1,k_2}[f] - R_{k_1,k,k_2}[f] - R_{k_2,k,k_1}[f] \Big] dk_1dk_2
 \end{equation}
with 
\begin{equation}\label{coll-op}
\begin{aligned}
R_{k,k_1,k_2} [f]:=   |V_{k,k_1,k_2}|^2\delta(k-k_1-k_2)\mathcal{L}_f(\omega_k -\omega_{k_1}-\omega_{k_2})(f_1f_2-ff_1-ff_2),
\end{aligned}
\end{equation}
with the short-hand notation $f = f(t,k)$ and $f_j = f(t,k_j)$. The singular measure given by the Dirac delta function $\delta(\cdot)$ ensures 
that interactions are between triads with 
\begin{equation}\label{cv} k = k_1 + k_2.\end{equation}

%
%

The  transition probability factor or collision kernel $V_{k,k_1,k_2}$ under consideration  is of the form (cf. \cite{lvov2004noisy,connaughton2001discreteness,lvov2004hamiltonian,lvov2012resonant,l1997statistical})
\begin{equation}\label{def-VV}
\begin{aligned}
V_{k,k_1,k_2}  \ = \mathfrak{C}\left({|k||k_1||k_2|}\right)^\frac12,
\end{aligned}
\end{equation}
with $\mathfrak{C}$ is some physical constant.

Next,  we consider the dispersion law 
\begin{equation}\label{Def:DispersionLaw}
\omega_k=\sqrt{F^2+\frac{g^2}{\rho_0^2N^2}\frac{|k|^2}{m^2}},
\end{equation}
where $F$  is the Coriolis parameter, $N$ is the (Brunt-Vaisala)  buoyancy frequency, In addition, the parameter $m$ is the reference vertical wave number determined from observations, $g$ is the gravitational constant, $\rho_0$ is the reference value for density, or equivalently 
\begin{equation}\label{Def:DispersionLaw2}
\omega_k=\sqrt{\lambda_1+\lambda_2 |k|^2}, \qquad \ \text{for} \ \lambda_1=F^2\, , \ \ \text{and}\  \ \lambda_2= \frac 1{m^2}\left(\frac{g}{\rho_0N}\right)^2=
 \frac 1{k_z^2}.
\end{equation}
where $k_z$ cartesian vertical wave number and $m=  k_z {g}{(\rho_0 N)}^{-1}$. 
In the absence of the Coriolis force, i.e. $F=0$,  the dispersion relation becomes
\begin{equation}\label{Def:DispersionLaw-WithoutCoriolis}
\omega_k= \frac{|k|}m \approx   \frac{|k|}{k_z} \, .
\end{equation}

The operator $\mathcal{L}_f$ is defined as
\begin{equation}
\label{OperatorL}
\mathcal{L}_f(\zeta)=\frac{\Gamma_{k,k_1,k_2}}{\zeta^2+\Gamma_{k,k_1,k_2}^2},
\end{equation}
with the condition that
$$\lim_{\Gamma_{k,k_1,k_2}\to 0}\mathcal{L}_f(\zeta)=\pi\delta(\zeta).$$
Thus when $\Gamma_{k,k_1,k_2}$ tends to $0$, \eqref{def-Qf} becomes the following exact resonance collision operator (cf. \cite{zakharov1967weak,zakharov1968stability,Hasselmann:1962:OTN1})
\begin{equation}\label{def-QfExact}Q_e[f](k) \ = \ \pi  \iint_{\mathbb{R}^{2d}} \Big[ \tilde{R}_{k,k_1,k_2}[f] - \tilde{R}_{k_1,k,k_2}[f] - \tilde{R}_{k_2,k,k_1}[f] \Big] dk_1dk_2 \end{equation}
with $$\begin{aligned}
\tilde{R}_{k,k_1,k_2} [f]:=  |V_{k,k_1,k_2}|^2\delta(k-k_1-k_2)\delta(\omega_k -\omega_{k_1}-\omega_{k_2})(f_1f_2-ff_1-ff_2). 
\end{aligned}
$$
With out loss of generality, one could ignore the constant $\pi$ in the collision operator $Q_e[f]$ since it can be absorbed in the time variable. 

Moreover, the resonance broadening frequency $\Gamma_{k,k_1,k_2}$ may be written
\begin{equation}
\label{Gamma}
\Gamma_{k,k_1,k_2}=\gamma_k+\gamma_{k_1}+\gamma_{k_2},
\end{equation}
where $\gamma_k$ is computed in \cite{l1997statistical} using a one-loop diagram approximation:
\begin{equation*}
\gamma_k\backsim \mathfrak{c}|k|^2\int_{\mathbb{R}_+}|k|^2|f(t,|k|)|d|k|,\end{equation*}
and $\mathfrak{c}$ is a physical constant, which can be normalized to be $1$.
Approximating the integral 
\begin{equation*}\int_{\mathbb{R}_+}|k|^2|f(t,|k|)|d|k| \approx \int_{\mathbb{R}^3}f(t,k)dk,\end{equation*} we obtain a formula for $\gamma_k$ that will be used throughout the paper
\begin{equation}\label{DampingTerm}
\gamma_k=|k|^2\int_{\mathbb{R}^3}f(t,k)dk.\end{equation}

The above formulation of $\gamma_{k}$ indicate the broadening resonance width $\theta$ defined in \eqref{NearResonantManifold}. Note that the formulation of $\Gamma_{k,k_1,k_2}$ is given
\begin{equation}\label{broaden1}
\Gamma_{k,k_1,k_2}=(|k|^2+|k_1|^2+|k_2|^2)\int_{\mathbb{R}^3}f(t,k)dk,
\end{equation}
Observe that 
$$\sqrt{n}\Gamma_{k,k_1,k_2} \le |\omega_k-\omega_{k_1}-\omega_{k_2}|\le \sqrt{n+1}\Gamma_{k,k_1,k_2}, \ \ \ n\in \mathbb{N},$$
then 
$$\frac{1}{(n+2)\Gamma_{k,k_1,k_2}} \le \mathcal{L}_f(\omega_k-\omega_{k_1}-\omega_{k_2})\le \frac{1}{(n+1)\Gamma_{k,k_1,k_2}},$$
in other words, function $\mathcal{L}_f(\omega_k-\omega_{k_1}-\omega_{k_2})$ is mostly concentrated in the interval where
\begin{equation}\label{broaden2}
|\omega_k-\omega_{k_1}-\omega_{k_2}|\le \Gamma_{k,k_1,k_2}.
\end{equation}
In other words, the resonance width $\theta$ is proportional to $\Gamma_{k,k_1,k_2}$, which depends on $f$ and $k$. 

This fact will be used in the proof of Propositions \ref{Propo:C12}, \ref{Propo:MassLowerBound} and \ref{Propo:HolderC12}.

In the field of wave turbulence, the most commonly used asymptotical analysis to derive the kinetic equation \eqref{WeakTurbulenceInitial}- \eqref{cv} is statistical closure of
the infinite hierarchy of cumulants, in the weakly nonlinear and long-time limits (see, for example, the review by Newell and Rumpf \cite{newell2011wave}).
Evolution of higher-order cumulants can be interpreted as a modification of the wave frequency, with real part corresponding to a frequency shift and
with imaginary part corresponding to resonance broadening. \\
A Feynman-Dyson diagrammatic approach may also be used, adapted for turbulence in fluids by Wyld \cite{wyld1961}, for more general classical systems by
Martin, Siggia and Rose  \cite{MartinSiggiaRose1973}, and for Hamiltonian nonlinear wave fields by Zakharov and Lvov \cite{ZakharovLvov1975}.   In the context of acoustic turbulence, 
Lvov, Lvov, Newell and Zakharov  \cite{l1997statistical} considered a one-loop approximation to the resonance broadening, the form of which is the one to be adopted in our study.

It is noted that  wave turbulence equation \eqref{WeakTurbulenceInitial} shares a  similar structure with the quantum Boltzmann equation describing the evolution of the excitations in thermal cloud Bose-Einstein condensate systems (cf. \cite{QK0,josserand2001nonlinear,KD1,MR1837939,PomeauBrachetMetensRica,ZakharovNazarenko:DOT:2005}). Our recent progress on  the classical Boltzmann equation (cf. \cite{Bobylev:BEF:2006,GambaPanferov:2004:OTB,GambaPanferovVillani:2009:UMB,AlonsoGamba:2015:OML}) and the quantum Boltzmann equation (cf. \cite{AlonsoGambaBinh,CraciunBinh,Binh9,germain2017optimal,JinBinh,ToanBinh,nguyen2017quantum,SofferBinh1,ReichlTran,SofferBinh2})  has shed some light on the open question of building a rigorous mathematical study for  \eqref{WeakTurbulenceInitial}. Different from the quantum Boltzmann cases (cf. \cite{SofferBinh1,AlonsoGambaBinh,CraciunBinh}), which could be considered as the exact resonance case \eqref{def-QfExact} with 
$$\omega_k=\omega_{k_1}+\omega_{k_2},$$
the energy of solutions for the  near-resonance  kinetic equation \eqref{WeakTurbulenceInitial} is not conserved.
The underlying shallow-water equations conserve a cubic energy, and the flow restricted to exact resonances conserves the quadratic part of the total energy \cite{warn86}.  However, 
conservation of the quadratic energy no longer holds when near resonant three-wave interactions are included in the dynamics.

We also split $Q$ as the sum of their positive and negative parts, referred to as a gain and a loss operators, respectively:
\begin{equation}\label{GainLoss}
Q[f] \ = \ Q_{\mathrm{gain}}[f] \ - \ Q_\mathrm{loss}[f],
\end{equation}
as is done with the classical Boltzmann operator  for binary elastic interactions. Here, the gain operator is also defined by the positive contributions in the total rate of change in time of the collisional form $Q(f)(t,k)$ 
\begin{equation}\label{Qgain}
\begin{aligned}
Q_\mathrm{gain}[f] \ = & \ \iint_{\mathbb{R}^{d}\times \mathbb{R}^{d}}|V_{k,k_1,k_2}|^2\delta(k-k_1-k_2)\mathcal{L}_f(\omega_k-\omega_{k_1}-\omega_{k_2})f_1f_2dk_1dk_2 \\
\ & +2\iint_{\mathbb{R}^{d}\times \mathbb{R}^{d}} |V_{k_1,k,k_2}|^2\delta(k_1-k-k_2)\mathcal{L}_f(\omega_{k_1}-\omega_{k}-\omega_{k_2})(ff_1+f_1f_2)dk_1dk_2.
\end{aligned}
\end{equation}
and the loss operator models the negative contributions in the total
rate of change in time of the same collisional form $Q(f)(t,k)$
\begin{equation}\label{Qlosss}Q_\mathrm{loss}[f] \ = \ f\vartheta[f],\end{equation}

with $\vartheta[f]$ being the collision frequency or attenuation coefficient, defined by
\begin{equation}\label{Qloss}
\begin{aligned}
\vartheta[f](k) \ = & \ 2\iint_{\mathbb{R}^{d}\times \mathbb{R}^{d}}|V_{k,k_1,k_2}|^2\delta(k-k_1-k_2)\mathcal{L}_f(\omega_k-\omega_{k_1}-\omega_{k_2})f_1dk_1dk_2 \\
\ & +2\iint_{\mathbb{R}^{d}\times \mathbb{R}^{d}} |V_{k_1,k,k_2}|^2\delta(k_1-k-k_2)\mathcal{L}_f(\omega_{k_1}-\omega_{k}-\omega_{k_2})f_2dk_1dk_2.
\end{aligned}
\end{equation}

Inspired by recent work by Alonso and two of the authors of the current manuscript \cite{AlonsoGambaBinh} on the quantum Boltzmann equation for cold bosonic gases, whose equation can also be derived by diagrammatic techniques,  we present here the existence and uniqueness solution to a Cauchy problem associated to  the model \eqref{WeakTurbulenceInitial}-\eqref{OperatorL}

The strategy consists in finding a suitable convex, positive cone, time invariant subspace {$S_T$}  of the Banach   space $L^1_N(\RR^d)$,  for which 
 the weak turbulence equation has a unique  strong solution, where this Banach space has norms defined by the $N^{th}$ moment as the expectation of the $N^{th}$-power of the dispersion relation, that is for any given density $g$,  
\begin{equation}\label{Def:MomentOrderk}
L^1_N(\RR^d):=\{g\in L^1(\mathbb{R}^d), \ \ s.t.\    \|g\|_{L^1_N} :=\mathfrak{M}_N[g]=\int_{\mathbb{R}^d}\omega^N_kg(k)dk \ <\infty \},
\end{equation}
in which we recall the dispersion relation $\cE_k = \sqrt{\lambda_1+\lambda_2|k|^2}$ as defined in \eqref{Def:DispersionLaw2}. 
Notice that when $g$ is positive, both  $\mathfrak{M}_n[g]$ and $\|g\|_{L^1_n}$ are equivalent.
Hence, the construction of such invariant subspace $S_T$  depend on the control of higher order moments defined as follows.

Our solution are  global and unique   in $L^1_N(\mathbb{R}^d)$ to \eqref{WeakTurbulenceInitial}, that is the satisfy 
\begin{equation}\label{WeakTurbulenceInitialReformed}
\partial_t f(t,k) \ = \ Q_\mathrm{gain}[f](t,k) \ - \ f(t,k)\vartheta[f](t,k) \ - 2\nu|k|^2f, \ \ \ f(0,k)=f_0(k) \in S_T.
\end{equation} 

A fundamental tool to accomplish our goal is to prove that there exists a differential equation of the following type, for the moments of the solution $f$ of \eqref{WeakTurbulenceInitialReformed}
$$\frac{d}{dt}\mathfrak{M}_N[f] \ \le \ C_1\mathfrak{M}_{N+1}[f]  - C_2\mathfrak{M}_{N+2}[f] ,$$
for some positive constants $C_1,C_2$, which leads to 
$$\frac{d}{dt}\mathfrak{M}_N[f] \ \le \ C_3\mathfrak{M}_{N}[f],$$
with $C_3$ being a positive constant. The above inequality then yields
 an exponential bound on the $N$-th moment of $f$
$$\mathfrak{M}_N[f] \le C e^{C'T}.$$

In order to do that, estimates on $Q_{\mathrm{gain}}$ and $Q_{\mathrm{loss}}$ are provided in Propositions \ref{Propo:C12} and \ref{Propo:MassLowerBound}. The proofs of these estimates are based on  careful bounds of $\mathcal{L}_f$ and $\Gamma_{k,k,k_1}$, that reduces to bounding the $0$-th moment of $f$, $\mathfrak{M}_0[f](t)$, from below by $e^{-(2\nu R_0^2+4R_0)t}\|f_0\chi_{R_0}\|_{L^1}$, where  $\chi_{R_0}$ is the characteristic function of the ball $B(O,R_0)$ centered at the origin with radius $R_0$ so that the quantity $\|f_0\chi_{R_0}\|_{L^1}>0$.

Finally, on any arbitrary fixed time interval $[0,T]$, we construct the solution of \eqref{WeakTurbulenceInitialReformed}  within a time-dependent invariant set $\mathcal{S}_T$, based on the exponential in time upper bound of $\mathfrak{M}_N[f]$ and the lower bound of $\mathfrak{M}_0[f]$. 

\bigskip

More specifically,  we define first the following two constants, $C^*$ and $C_*$,    for any given any $R_0$ by 
 \begin{equation}\label{Cstars}
 {C_* }: = \frac{C_0({\lambda_1},{\lambda_2})\left(1+{e^{(4\nu R_0^2+8R_0)T}}\right)}{{\|f_0(k)\chi_{R_0}\|_{L^1}}}, \ \ \  \text{and}\ \ \ {C^* : = 4\nu R_0^2+8R_0. }
  \end{equation}
  The specific value of $R_0$ will be determined later to secure the conditions to obtain a time invariant region.

 Hence, for any number $R^*>0$, $R_*>1$, moment order  $N$, and time $t >0$, we define the  convex positive cone $S_T$   as a subset on $L^1_N$ given by 
 \begin{align}\label{subsetS_T}
 \mathcal{S}_T &:= \Big{\{ }f \in   L^{1}_{N+3}\big(\mathbb{R}^{d}\big) \ : \   
{\bf S1}) \ f\ge 0; \ 
\ \  {\bf S2})  \   \| f\|_{L^1_{N+3}} \le  {{c}_0(t)}:=(2{R}_*+1)e^{{C_*}t};\\
&\qquad\ \ \ {\bf S3})  \ \| f\|_{L^1} \ge {{c}_{1}(t)}:=\frac{R^*e^{{-C^*}t}}{2}.
\Big{\} \, } \nonumber
   \end{align}
where  the $c_0(t)$ is an increasing function and $c_1(t)$ is a decreasing function, so $S_t\subset S_{t'}$ for $0\le t\le t'\le T$

 Our main result is as follows.


%
%
%
%

\begin{theorem}\label{Theorem:Main}
Let $N>0$, and let $f_0(k) \in \mathcal{S}_0\cap  B_*(O,R_*)\backslash\overline{B_*(O,R^*)}$ for some $R^*>R_*>0$, where $B_*(O,R^*),B_*(O,R_*)$ is the ball centered at $O$ with radius $R^*,R_*$ of $L^1_{N+3}(\mathbb{R}^d)$.

Then the weak turbulence equation \eqref{WeakTurbulenceInitial} has a unique  strong solution $f(t,k)$ so that
\begin{equation}\label{the_theorem}
0\leq f(t,k)\in {C}\left([0,T); L^1_N(\RR^d)\right)\cap  {C}^1\left((0,T);L^1_N(\RR^d)\right).
\end{equation}

Moreover, $f(t,k)\in \mathcal{S}_T$ for all $t\in [0,T)$. 

Since $T$ can be chosen arbitrarily large, the weak turbulence equation \eqref{WeakTurbulenceInitial} has a unique global solution for all time $t>0$. 
\end{theorem}

The proof of Theorem \ref{Theorem:Main} relies on the following abstract Ordinary Differential Equations theorem in Banach spaces, which provides a framework to developed the  existence and uniqueness theory to space homogeneous Boltzmann type equations ranging from the classical Boltzmann equation for binary interaction, to  nonlocal kinetic model for rods alignment, to  quantum kinetic theory of bosonic cold gases \cite{Bressan, ABCL:2016, AlonsoGambaBinh, AlonsoGamba2018}.

 Applied to the initial value problem \eqref{WeakTurbulenceInitial}-\eqref{broaden2}, the framework is given by the following abstract existence and uniqueness theorem in Banach spaces along the lines proposed by A. Bressan in the unpublished notes \cite{Bressan}, whose application to the classical Boltzmann theory for hard potential and integrable angular cross section has been recently completed in \cite{AlonsoGamba2018}, as follows.
 
 Let     $E:=(E,\|\cdot\|)$ be a Banach space of real functions on $\mathbb{R}^d$, $(F,\|\cdot\|_*)$ be a Banach subspace of $E$ satisfying $\|u\|\le \|u\|_*$ $\forall u\in F$.  Denote by $B(O,r)$, $B_*(O,r)$   the balls centered at $O$ with radius $r>0$ with respect to the norm $\|\cdot\|$ and $\|\cdot\|_*$. Suppose that there exists a function $|\cdot|_*$ from $F$ to $\mathbb{R}$ such that
$$|u|_* \le \|u\|_*, \ \ \forall u\in F,\ \ \ |u+v|_* \le |u|_*+|v|_*, \ \ \forall u,v\in F,$$
$$\lambda|u|_* = |\lambda u|_*, \ \ \forall u\in F,\lambda\in\mathbb{R}_+.$$
where $C$ is some positive constant.  

\begin{theorem}\label{Theorem:ODE} Let $[0,T]$ be a time interval, and $\mathcal{S}_t$, $(t\in[0,T])$,   be a class of bounded and closed subset of $F$satisfying $S_t\subset S_{t'}$ for $0\le t\le t'$ and containing only non-negative functions and  
$$|u|_* = \|u\|_*, \ \ \forall u\in \mathcal{S}_T.$$
Moreover, for any sequence $\{u_n\}$ in $S_T$, 
\begin{equation}\label{LesbegueDominated}
\mbox{ If } u_n\geq 0, \|u_n\|_*\le C, \lim_{n\to\infty} \|u_n-u\| = 0, \mbox{ then } \lim_{n\to\infty} \|u_n-u\|_* = 0,
\end{equation}
Set $R_*>R^*>0$ and suppose  $\mathcal{Q}:\mathcal{S}_T\rightarrow E$ is an operator  satisfying the following properties: There exist $R_0,C_*,C^*>0$ such that: 
\begin{itemize}

\item [$(\mathfrak{A})$] H\"{o}lder continuity condition
\begin{equation*}
\big\|Q[u] - Q[v]\big\| \leq C\|u - v\|^{\beta},\quad \beta\in(0,1), \quad \forall\,u,v\in\mathcal{S}_T\,.
\end{equation*}
\item [$(\mathfrak{B})$] Sub-tangent condition

 For an element $u$ in $\mathcal{S}_T$,  there exists $\xi_u>0$ such that for $0<\xi<\xi_u$, there exists $z$ in $B(u+\xi \mathcal{Q}[u],\delta)\cap \mathcal{S}_T\backslash \{u+\xi\mathcal{Q}[u]\}$ for $\delta$ small enough. Moreover, 
\begin{equation}\label{Coercivity}\begin{aligned}
|z-u|_* &\ \le \frac{C_*\xi}{2}\|u\|_*,\\
\chi_{R_0}\frac{z-u}{\xi} &\ \geq -\frac{C^*\chi_{R_0}}{2}u,
\end{aligned}
\end{equation}
where $\chi_{R_0}$ is the characteristic function of the ball $B_{\mathbb{R}^d}(0,R_0)$ of $\mathbb{R}^d$.

\item [$(\mathfrak{C})$] one-side Lipschitz condition
\begin{equation*}
\big[ Q[u] - Q[v], u - v \big] \leq C\|u - v\|,\qquad \forall\,u,v\in\mathcal{S}_T\,,
\end{equation*}
where $$\big[ \varphi,\phi \big]: = \lim_{h\rightarrow 0^{-}}h^{-1}\big(\| \phi + h\varphi \| - \| \phi \| \big).$$

\end{itemize}
Moreover, $\mathcal{S}_T\cap B\left(0,\frac{R^*e^{-C^*T}}{2}\right)=\emptyset$ and $\mathcal{S}_T\subset B(0,(2R_*+1)e^{C_*T})$.

Then the equation 
\begin{equation}\label{Theorem_ODE_Eq}
\partial_t u=Q[u] \mbox{ on } [0,T)\times E,~~~~u(0)=u_0 \in \mathcal{S}_0\cap B_*(O,R_*)\backslash\overline{B_*(O,R^*)}, 
\end{equation}
 has a unique solution $$u \in C^1((0,T),E)\cap C\left([0,T),\mathcal{S}_T\right).$$
\end{theorem}

We end this introduction by giving the structure of the paper. In Section \ref{Sec:MomentPropa}, we provide an a priori estimate on the $L^1_N$ norm of the solution.  The H\"older continuity of the collision operator will be established in Section \ref{Sec:HolderEstimate}. The proof of Theorem \ref{Theorem:Main} is given in Section \ref{Sec:Main}. The proof of Theorem \ref{Theorem:ODE} is given in Section 5.

Throughout the paper, we normally denote by $C$, $C'$ universal constants that may vary from line to line.


%
%
%

\section{A priori estimate}\label{Sec:MomentPropa}
In this section, we shall derive uniform estimates on the $N$-th moment of $f$.
\subsection{Preliminaries}
The following lemma represents the weak formulation for the collision operator
\begin{lemma}\label{Lemma:WeakFormulation}
There holds 
$$
\begin{aligned}
\int_{\RR^d}Q[f](t,k) \varphi(k) \; dk  \ = &\  \iiint_{\mathbb{R}^{3d}} R_{k,k_1,k_2} [f]  \Big[ \varphi(k) - \varphi(k_1) - \varphi(k_2)\Big] dk dk_1dk_2 
\end{aligned}
$$
for any test functions $\varphi$ so that the integrals are well-defined. 
\end{lemma}
\begin{proof} By definition, the  integral of the product of $Q[f]$ and $\varphi$ is written
$$
\begin{aligned}
\int_{\RR^d}Q[f](t,k) \varphi(k) \; dk  \ = &\ \iiint_{\mathbb{R}^{3d}} \Big[ R_{k,k_1,k_2} - R_{k_1,k,k_2} - R_{k_2,k,k_1} \Big] \varphi(k) dk dk_1dk_2 .
\end{aligned}
$$
By employing the change of variables $k\leftrightarrow k_1$, $k\leftrightarrow k_2$ in the first integral on the right, the lemma then follows.  
\end{proof}
In this paper, we also need the following H\"older-type inequality.
\begin{lemma}\label{lem-Holder} For $N>n>p$, and $g\geq 0$ there holds
\begin{equation}\label{Holder} \mathfrak{M}_n[g] \le \mathfrak{M}_p^{\frac{N-n}{N-p}}[g] \mathfrak{M}_{N}^{\frac{n-p}{N-p}}[g],\end{equation}
where $g$ is such that all of the integrals are well-defined. 
\end{lemma}
\begin{proof}
The lemma follows from the definition of $\mathfrak{M}_n$ and the following H\"older inequality
\begin{eqnarray*}
\int_{\mathbb{R}^d} g(k)\cE_k^n dk &\le& \left( \int_{\mathbb{R}^d} g(k)\cE_k^p dk \right)^{\frac{N-n}{N-p}} \left( \int_{\mathbb{R}^d} g(k)\cE_k^N dk\right)^{\frac{n-p}{N-p}} .
\end{eqnarray*}
\end{proof}

\subsection{Estimate of the collision operator}
The main result of this subsection is the following estimate on the gain part of the collision operator $Q[g]$ as defined in \eqref{GainLoss} and \eqref{Qgain}. 

\begin{lemma}\label{Propo:C12} Let $N\geq 0 $. For any positive  function $g\in L^1_{N+1}$, there exists a constant $\mathcal{C}C(\lambda_1,\lambda_2,N)$, depending on $\lambda_1,\lambda_2,N$, such that the following holds 
\begin{equation}\label{Propo:C12:1}
 \int_{\mathbb{R}^d}Q_{\mathrm{gain}}[g](k)\cE^N_k \; dk\le   \frac{\mathcal{C}(\lambda_1,\lambda_2,N)\mathfrak{M}_{N+1}[g]}{\mathfrak{M}_0[g]}.
\end{equation}
\end{lemma}

\begin{remark}
The proof below is based on the fact that the resonance broadening width $\theta$ defined in \eqref{NearResonantManifold} is chosen proportional to 
$$(|k|^2+|k_1|^2+|k_2|^2)\int_{\mathbb{R}^3}f(t,k)dk,$$
as discussed in the introduction.
\end{remark}
\begin{proof} By the same argument used to obtain the weak formulation proved in Lemma \ref{Lemma:WeakFormulation}, the following identity holds true
$$
\begin{aligned}
\int_{\RR^d}Q[g](k) \cE_k^N \; dk  \ = &\  \iiint_{\mathbb{R}^{3d}} \tilde{R}_{k,k_1,k_2} [g]  \Big [\cE^N_{k}-\cE^N_{k_1}-\cE^N_{k_2} \Big] dk dk_1dk_2, 
\end{aligned}
$$
where
$$\tilde{R}_{k,k_1,k_2} [g]:=   |V_{k,k_1,k_2}|^2\delta(k-k_1-k_2)\mathcal{L}(\omega_k -\omega_{k_1}-\omega_{k_2})(g_1g_2+gg_1+gg_2).$$

And the integration of the gain term in multiplying with the test function $\cE^N_{k}$ is then
$$
\begin{aligned}
& \ \int_{\RR^d}Q_{\mathrm{gain}}[g](k) \cE_k^N \; dk=  \\
= &\  C\iiint_{\mathbb{R}^{3d}} \delta(k-k_1-k_2)\frac{\mathfrak{M}_0[g](|k|^2+|k_1|^2+|k_2|^2)|k||k_1||k_2|}{(\omega_k -\omega_{k_1}-\omega_{k_2})^2+\mathfrak{M}_0[g]^2(|k|^2+|k_1|^2+|k_2|^2)^2}\times\\
 &\ \times g_1g_2\cE^N_{k}dk dk_1dk_2\\
 &\  + C\iiint_{\mathbb{R}^{3d}} \delta(k_1-k-k_2)\frac{\mathfrak{M}_0[g](|k|^2+|k_1|^2+|k_2|^2)|k||k_1||k_2|}{(\omega_{k_1} -\omega_{k}-\omega_{k_2})^2+\mathfrak{M}_0[g]^2(|k|^2+|k_1|^2+|k_2|^2)^2}\times\\
 &\ \times (gg_1+g_1g_2)  \cE^N_{k}dk dk_1dk_2,
\end{aligned}
$$
which, by the change of variable $(k,k_1)\to (k_1,k)$ in the second integral, whose Jacobian is 1, could be expressed as
$$
\begin{aligned}
& \ \int_{\RR^d}Q_{\mathrm{gain}}[g](k) \cE_k^N \; dk=  \\
= &\  C\iiint_{\mathbb{R}^{3d}} \delta(k-k_1-k_2)\frac{\mathfrak{M}_0[g](|k|^2+|k_1|^2+|k_2|^2)|k||k_1||k_2|}{(\omega_k -\omega_{k_1}-\omega_{k_2})^2+\mathfrak{M}_0[g]^2(|k|^2+|k_1|^2+|k_2|^2)^2}\times\\
 &\ \times g_1g_2\cE^N_{k}dk dk_1dk_2\\
 &\  + C\iiint_{\mathbb{R}^{3d}} \delta(k-k_1-k_2)\frac{\mathfrak{M}_0[g](|k|^2+|k_1|^2+|k_2|^2)|k||k_1||k_2|}{(\omega_{k} -\omega_{k_1}-\omega_{k_2})^2+\mathfrak{M}_0[g]^2(|k|^2+|k_1|^2+|k_2|^2)^2}\times\\
 &\ \times (gg_1+gg_2)  \cE^N_{k_1}dk dk_1dk_2.
\end{aligned}
$$

By the symmetry of $k_1$ and $k_2$ in the second integral, 
$$
\begin{aligned}
& \ \int_{\RR^d}Q_{\mathrm{gain}}[g](k) \cE_k^N \; dk=  \\
= &\  C\iiint_{\mathbb{R}^{3d}} \delta(k-k_1-k_2)\frac{\mathfrak{M}_0[g](|k|^2+|k_1|^2+|k_2|^2)|k||k_1||k_2|}{(\omega_k -\omega_{k_1}-\omega_{k_2})^2+\mathfrak{M}_0[g]^2(|k|^2+|k_1|^2+|k_2|^2)^2}\times\\
 &\ \times g_1g_2\cE^N_{k}dk dk_1dk_2\\
 &\  + C\iiint_{\mathbb{R}^{3d}} \delta(k-k_1-k_2)\frac{\mathfrak{M}_0[g](|k|^2+|k_1|^2+|k_2|^2)|k||k_1||k_2|}{(\omega_{k} -\omega_{k_1}-\omega_{k_2})^2+\mathfrak{M}_0[g]^2(|k|^2+|k_1|^2+|k_2|^2)^2}\times\\
 &\ \times gg_1\left[\cE^N_{k_1}+\cE^N_{k_2}\right]dk dk_1dk_2.
\end{aligned}
$$

Let us now look at the fractional term in the above integral
$$K:=\frac{\mathfrak{M}_0[g](|k|^2+|k_1|^2+|k_2|^2)|k||k_1||k_2|}{(\omega_k -\omega_{k_1}-\omega_{k_2})^2+\mathfrak{M}_0[g]^2(|k|^2+|k_1|^2+|k_2|^2)^2}.$$
Since the denominator $(\omega_k -\omega_{k_1}-\omega_{k_2})^2+\mathfrak{M}_0^2(|k|^2+|k_1|^2+|k_2|^2)^2$ is greater than $\mathfrak{M}_0[g]^2(|k|^2+|k_1|^2+|k_2|^2)^2$, the whole fraction can be bounded as
 $$K\le \frac{|k||k_1||k_2|}{\mathfrak{M}_0[g](|k|^2+|k_1|^2+|k_2|^2)},$$
which leads to the following 
$$
\begin{aligned}
& \ \int_{\RR^d}Q_{\mathrm{gain}}[g](k) \cE_k^N \; dk  \\
 \le &\  C\iiint_{\mathbb{R}^{3d}} \delta(k-k_1-k_2)\frac{|k||k_1||k_2|}{\mathfrak{M}_0[g](|k|^2+|k_1|^2+|k_2|^2)}g_1g_2\cE^N_{k} dk dk_1dk_2\\
 & + \  C\iiint_{\mathbb{R}^{3d}} \delta(k-k_1-k_2)\frac{|k||k_1||k_2|}{\mathfrak{M}_0[g](|k|^2+|k_1|^2+|k_2|^2)}gg_1\Big [\cE^N_{k_1}+\cE^N_{k_2} \Big] dk dk_1dk_2,
\end{aligned}
$$
which can be rewritten in the following equivalent form, with the right hand side being the sum of $I_1$ and $I_2$
\begin{equation}\label{Propo:C12:E1}
\begin{aligned}
 \int_{\RR^d}Q_{\mathrm{gain}}[g](k) \cE_k^N \; dk  
 \le &\  I_1+I_2,
\end{aligned}
\end{equation}
where
\begin{equation}\label{Propo:C12:E2}
\begin{aligned}
I_1 := & \ C\iiint_{\mathbb{R}^{3d}} \delta(k-k_1-k_2)\frac{|k||k_1||k_2|}{\mathfrak{M}_0[g](|k|^2+|k_1|^2+|k_2|^2)}g_1g_2  \cE^N_{k}dk dk_1dk_2\\
I_2 := & \ C\iiint_{\mathbb{R}^{3d}} \delta(k-k_1-k_2)\frac{|k||k_1||k_2|}{\mathfrak{M}_0[g](|k|^2+|k_1|^2+|k_2|^2)}gg_1  \Big [\cE^N_{k_1}+\cE^N_{k_2} \Big] dk dk_1dk_2.
\end{aligned}
\end{equation}
Let us first estimate $I_1$. By the resonant condition  $k = k_1 + k_2$,  we have
$$\cE_k =\sqrt{\lambda_1+\lambda_2 |k|^2} \le \sqrt{\lambda_1+\lambda_2 (|k_1|+|k_2|)^2}< 2\sqrt{\lambda_1+\lambda_2 |k_1|^2} + 2\sqrt{\lambda_1+\lambda_2 |k_2|^2} = 2\cE_{k_1} +2\cE_{k_2},$$ 
which, thanks to the Cauchy-Schwarz inequality, leads to
$$\cE_k^N  \le C(\lambda_1,\lambda_2,N)(\cE_{k_1}^N + \cE_{k_2}^N),$$ 
where $C(\lambda_1,\lambda_2,N)$ is some  constant depending on $\lambda_1,\lambda_2,N$.\\
Thus, we obtain 
$$
\begin{aligned}
I_1 \le & \ C(\lambda_1,\lambda_2,N)\iiint_{\mathbb{R}^{3d}} \delta(k-k_1-k_2)\frac{|k||k_1||k_2|}{\mathfrak{M}_0[g](|k|^2+|k_1|^2+|k_2|^2)}g_1g_2  \Big [\cE^N_{k_1}+\cE^N_{k_2} \Big] dk dk_1dk_2.
\end{aligned}
$$
Taking into account the definition of the Dirac function $\delta(k-k_1-k_2)$ the above integral on $\mathbb{R}^{3d}$ can be reduced to an integral on $\mathbb{R}^{2d}$ only
$$
\begin{aligned}
I_1 \le & \ C(\lambda_1,\lambda_2,N)\iint_{\mathbb{R}^{2d}} \frac{|k_1+k_2||k_1||k_2|}{\mathfrak{M}_0[g](|k|^2+|k_1|^2+|k_2|^2)}g_1g_2  \Big [\cE^N_{k_1}+\cE^N_{k_2} \Big] dk_1dk_2.
\end{aligned}
$$
Due to the inequality $|k_1+k_2|^2+|k_1|^2+|k_2|^2\ge 2 |k_1||k_2|$, the kernel of the above integral can be bounded as
$$\frac{|k_1+k_2||k_1||k_2|}{|k_1+k_2|^2+|k_1|^2+|k_2|^2}\le \frac{|k_1+k_2|}{2} \le \frac{|k_1|+|k_2|}{2},$$
yielding 
$$
\begin{aligned}
I_1 \le & \ \frac{C(\lambda_1,\lambda_2,N)}{\mathfrak{M}_0[g]}\iint_{\mathbb{R}^{2d}} {(|k_1|+|k_2|)}g_1g_2  \Big [\cE^N_{k_1}+\cE^N_{k_2} \Big] dk_1dk_2.
\end{aligned}
$$
Observing that
$$|k_1|\le \frac{\cE_{k_1}}{\sqrt{\lambda_2}}, \ \ \  |k_2|\le \frac{\cE_{k_2}}{\sqrt{\lambda_2}},$$
we can bound
$$(|k_1|+|k_2|)\Big [\cE^N_{k_1}+\cE^N_{k_2} \Big]\le C(\cE_{k_1}+\cE_{k_2})\Big [\cE^N_{k_1}+\cE^N_{k_2} \Big] \le C\Big [\cE^{N+1}_{k_1}+\cE^{N+1}_{k_2} \Big],$$
which yields the following estimate on $I_1$ in terms of  the functional defined in \eqref{Def:MomentOrderk}
\begin{equation}\label{Propo:C12:E3}
\begin{aligned}
I_1 \le & \ \frac{C(\lambda_1,\lambda_2,N)}{\mathfrak{M}_0[g]}\iint_{\mathbb{R}^{2d}} g_1g_2 \Big [\cE^{N+1}_{k_1}+\cE^{N+1}_{k_2} \Big] dk_1dk_2\\ 
\le & \ \frac{C}{\mathfrak{M}_0[g]}\mathfrak{M}_{N+1}[g].
\end{aligned}
\end{equation}
Let us now estimate $I_2$. 
Using the resonant condition  $k_2 = k- k_1$,  we obtain the following relation between $\omega_{k_2}$ and $\omega_k$, $\omega_{k_1}$
$$\cE_{k_2} =\sqrt{\lambda_1+\lambda_2 |k_2|^2} < \sqrt{\lambda_1+\lambda_2 (|k_1|+|k|)^2}\le 2\sqrt{\lambda_1+\lambda_2 |k|^2} + 2\sqrt{\lambda_1+\lambda_2 |k_1|^2} = 2\cE_{k} +2\cE_{k_1},$$ 
which, by the Cauchy-Schwarz inequality, leads to
$$\cE_{k_2}^N  \le C(\lambda_1,\lambda_2,N)(\cE_{k}^N + \cE_{k_1}^N),$$ 
where $C$ is some universal positive constant.\\
Thus, we obtain 
$$
\begin{aligned}
I_2 \le & \ C(\lambda_1,\lambda_2,N)\iiint_{\mathbb{R}^{3d}} \delta(k-k_1-k_2)\frac{|k||k_1||k_2|}{\mathfrak{M}_0[g](|k|^2+|k_1|^2+|k_2|^2)}gg_1  \Big [\cE^N_{k}+\cE^N_{k_1} \Big] dk dk_1dk_2.
\end{aligned}
$$
By the definition of the Dirac function $\delta(k-k_1-k_2)$, we can reduce  the above triple integral into an integral on $\mathbb{R}^{2d}$ only
$$
\begin{aligned}
I_2 \le & \ C(\lambda_1,\lambda_2,N)\iint_{\mathbb{R}^{2d}} \frac{|k||k_1||k-k_1|}{\mathfrak{M}_0[g](|k|^2+|k_1|^2+|k_2|^2)}gg_1  \Big [\cE^N_{k}+\cE^N_{k_1} \Big] dk dk_1.
\end{aligned}
$$
It is straightforward from Cauchy-Schwarz inequality that $|k|^2+|k_1|^2+|k-k_1|^2\ge 2 |k_1||k|$, yielding the following estimate on the kernel of the above integral 
$$\frac{|k||k_1||k-k_1|}{|k|^2+|k_1|^2+|k-k_1|^2}\le \frac{|k-k_1|}{2} \le \frac{|k|+|k_1|}{2},$$
which implies the following bound on $I_2$ 
$$
\begin{aligned}
I_2 \le & \ \frac{C(\lambda_1,\lambda_2,N)}{\mathfrak{M}_0[g]}\iint_{\mathbb{R}^{2d}} {(|k|+|k_1|)}gg_1  \Big [\cE^N_{k}+\cE^N_{k_1} \Big] dk dk_1.
\end{aligned}
$$
The same argument used to estimate $I_1$ can now be applied again, that leads to a similar bound on $I_2$
\begin{equation}\label{Propo:C12:E4}
\begin{aligned}
I_2 \le & \ \frac{C(\lambda_1,\lambda_2,N)}{\mathfrak{M}_0[g]}\iint_{\mathbb{R}^{2d}} gg_1 \Big [\cE^{N+1}_{k}+\cE^{N+1}_{k_1} \Big] dk dk_1\\ 
\le & \ \frac{C(\lambda_1,\lambda_2,N)}{\mathfrak{M}_0[g]}\mathfrak{M}_{N+1}[g].
\end{aligned}
\end{equation}
Combining \eqref{Propo:C12:E1}, \eqref{Propo:C12:E2}, \eqref{Propo:C12:E3} and \eqref{Propo:C12:E4}, we get \eqref{Propo:C12:1} so the conclusion of the Lemma \ref{Propo:C12} follows. 
\end{proof}

\subsection{Lower bound of the solution (the choice of $R_0$)}
\begin{proposition}\label{Propo:MassLowerBound} For any initial data $f_0\geq 0$ and $f_0\in L^1(\mathbb{R}^3)$. Suppose that $f\in L^1(\mathbb{R}^d)$ is a positive, strong solution of \eqref{WeakTurbulenceInitial}, then 
\begin{equation}\label{Propo:MassLowerBound:1} Q[f] = \ Q_\mathrm{gain}[f] - Q_\mathrm{loss}[f] \geq -Q_\mathrm{loss}[f] \ \geq  \ -{4|k|}f,
\end{equation} 
pointwise in $k$ and $f$ satisfies the following lower bound \begin{equation}\label{Propo:MassLowerBound:2}  f(t,k) \geq f_0(k)e^{-(2\nu |k|^2+4|k|)t},\end{equation}
which implies
\begin{equation}\label{Propo:MassLowerBound:2a}\|f(t,k)\chi_{R_0}\|_{L^1}\ge \tilde{\mathfrak{M}}_0(t):= e^{-(2\nu R_0^2+4R_0)t} \|f_0(k)\chi_{R_0}\|_{L^1},\end{equation}
where $\chi_{R_0}$ is the characteristic function of the ball $B_{\mathbb{R}^d}(O,R_0)$ in $\mathbb{R}^d$, $R_0$ is any positive constant. 

\end{proposition}
\begin{proof}
Let us first recall the formulation of $Q[f]$
\begin{equation*}
\begin{aligned}
Q[f] \ = & \ \iint_{\mathbb{R}^{d}\times \mathbb{R}^{d}}|V_{k,k_1,k_2}|^2\delta(k-k_1-k_2)\mathcal{L}_f(\omega_k-\omega_{k_1}-\omega_{k_2})(f_1f_2-2ff_1)dk_1dk_2 \\
\ & +2\iint_{\mathbb{R}^{d}\times \mathbb{R}^{d}} |V_{k_1,k,k_2}|^2\delta(k_1-k-k_2)\mathcal{L}_f(\omega_{k_1}-\omega_{k}-\omega_{k_2})(-ff_2+ff_1+f_1f_2)dk_1dk_2.
\end{aligned}
\end{equation*}
and in order to get \eqref{Propo:MassLowerBound:2}, we will work with 
$$Q[f] \ = \ Q_{\mathrm{gain}}[f] \ - \ Q_{\mathrm{loss}}[f],$$
where the formulation of $Q_\mathrm{loss}[f]$
\begin{equation}\label{Propo:MassLowerBound:E1}
\begin{aligned}
-Q_\mathrm{loss}[f] \ =  & \ -2f\int_{\mathbb{R}^{d}\times \mathbb{R}^{d}}|V_{k,k_1,k_2}|^2\delta(k-k_1-k_2)\mathcal{L}_f(\omega_k-\omega_{k_1}-\omega_{k_2})f_1dk_1dk_2 \\
\ & -2f\int_{\mathbb{R}^{d}\times \mathbb{R}^{d}} |V_{k_1,k,k_2}|^2\delta(k_1-k-k_2)\mathcal{L}_f(\omega_{k_1}-\omega_{k}-\omega_{k_2})f_2dk_1dk_2\\
\ =:& \  -\mathcal{I}_1 - \mathcal{I}_2.
\end{aligned}
\end{equation}

In order to get the lower bound \eqref{Propo:MassLowerBound:1}, we discard the gain operator defined in \eqref{Qgain} and estimate from below the loss part.

Let us estimate the double integral $\mathcal{I}_1$, which can be reduced to an integral on $\mathbb{R}^d$ by taking into account the definition of $\delta(k-k_1-k_2)$ as follows
$$
\begin{aligned}
\mathcal{I}_1 \ :=  & \ 2f\int_{\mathbb{R}^{d}}|V_{k,k_1,{k-k_1}}|^2\mathcal{L}_f(\omega_k-\omega_{k_1}-\omega_{k-k_1})f_1dk_1.
\end{aligned}
$$
By the definition of $V_{k,k_1,{k-k_1}}$, $\mathcal{L}_f(\omega_k-\omega_{k_1}-\omega_{k-k_1})$, $\Gamma_{k,k_1,k_2}$, and the inequality $$(\omega_k-\omega_{k_1}-\omega_{k-k_1})^2+\Gamma_{k,k_1,k-k_1}^2\ge \Gamma_{k,k_1,k-k_1}^2,$$ 
we obtain the following inequality on the kernel of $\mathcal{I}_1$
$$
\begin{aligned}
|V_{k,k_1,{k-k_1}}|^2\mathcal{L}_f(\omega_k-\omega_{k_1}-\omega_{k-k_1}) \ = & \ \frac{|k||k_1||k-k_1|\Gamma_{k,k_1,k-k_1}}{(\omega_k-\omega_{k_1}-\omega_{k-k_1})^2+\Gamma_{k,k_1,k-k_1}^2}\\
\le & \ \frac{|k||k_1||k-k_1|}{\Gamma_{k,k_1,k-k_1}}\\
\le & \ \frac{|k||k_1||k-k_1|}{\mathfrak{M}_0[f](|k|^2+|k_1|^2+|k-k_1|^2)}.
\end{aligned}
$$
By the positivity of $|k|^2$ and the Cauchy-Schwarz inequality, the following holds true
$$|k|^2+|k_1|^2+|k-k_1|^2 \ge |k_1|^2+|k-k_1|^2 \ge 2 |k_1||k-k_1|,$$
which implies
$$
\begin{aligned}
|V_{k,k_1,{k-k_1}}|^2\mathcal{L}_f(\omega_k-\omega_{k_1}-\omega_{k-k_1}) \
\le & \ \frac{2|k|}{\mathfrak{M}_0[f]}.
\end{aligned}
$$
As a result, we have the following estimate on $\mathcal{I}_1$
\begin{equation}\label{Propo:MassLowerBound:E2}
\begin{aligned}
\mathcal{I}_1 \ \le  & \ \frac{2|k|f\int_{\mathbb{R}^d}f_1dk_1}{\mathfrak{M}_0[f]} \ \le \ {2|k|}f.
\end{aligned}
\end{equation}
$\mathcal{I}_2$  can be estimated in a similar way. We can reduce $\mathcal{I}_2$ to an integral on $\mathbb{R}^d$ by taking into account the definition of $\delta(k_1-k-k_2)$ as follows
$$
\begin{aligned}
\mathcal{I}_2 \ :=  & \ f\int_{\mathbb{R}^{d}}|V_{k+k_2,k,{k_2}}|^2\mathcal{L}_f(\omega_{k+k_2}-\omega_{k}-\omega_{k_2})f_2dk_2.
\end{aligned}
$$
Taking into account the definite of $V_{k+k_2,k,{k_2}}$, $\mathcal{L}_f(\omega_{k+k_2}-\omega_{k}-\omega_{k_2})$, $\Gamma_{k+k_2,k,{k_2}}$, and the inequality $$(\omega_{k+k_2}-\omega_{k}-\omega_{k_2})^2+\Gamma_{k+k_2,k,{k_2}}^2\ge \Gamma_{k+k_2,k,{k_2}}^2,$$ 
 the following estimate on the kernel of $\mathcal{I}_2$ can be obtained
$$
\begin{aligned}
|V_{k+k_2,k,{k_2}}|^2\mathcal{L}_f(\omega_{k+k_2}-\omega_{k}-\omega_{k_2})\ = & \ \frac{|k+k_2||k||k_2|\Gamma_{k+k_2,k,{k_2}}}{(\omega_{k+k_2}-\omega_{k}-\omega_{k_2})^2+\Gamma_{k+k_2,k,{k_2}}^2}\\
\le & \ \frac{|k+k_2||k||k_2|}{\mathfrak{M}_0[f](|k+k_2|^2+|k|^2+|k_2|^2)}.
\end{aligned}
$$
Using  the positivity of $|k|^2$ and the Cauchy-Schwarz inequality, we find
$$|k+k_2|^2+|k|^2+|k_2|^2 \ge |k+k_2|^2+|k_2|^2 \ge 2 |k+k_2||k_2|,$$
which implies
$$
\begin{aligned}
|V_{k+k_2,k,{k_2}}|^2\mathcal{L}_f(\omega_{k+k_2}-\omega_{k}-\omega_{k_2}) \
\le & \ \frac{2|k|}{\mathfrak{M}_0[f]}.
\end{aligned}
$$
We then obtain the following estimate on $\mathcal{I}_2$
\begin{equation}\label{Propo:MassLowerBound:E3}
\begin{aligned}
\mathcal{I}_2 \ \le  & \ \frac{2|k|f\int_{\mathbb{R}^d}f_2dk_2}{\mathfrak{M}_0[f]} \ = \ {2|k|}f.
\end{aligned}
\end{equation}
Combining \eqref{Propo:MassLowerBound:E1}, \eqref{Propo:MassLowerBound:E2} and \eqref{Propo:MassLowerBound:E3} yields
\begin{equation}\label{Propo:MassLowerBound:E4}
\begin{aligned}
Q[f] \ \geq  & \ -{4|k|}f.
\end{aligned}
\end{equation}
By plugging the above inequality into \eqref{WeakTurbulenceInitial}, we obtain a differential inequality on $f$
$$\partial_tf -Q[f] - 2\nu|k|^2f\ \ge \ \partial_t f \ + \ (2\nu |k|^2+4|k|)f  \ \ge \ 0.$$
A Gronwall inequality argument applied to the above differential inequality leads to
$$f(t,k) \geq f_0(k)e^{-(2\nu |k|^2+4|k|)t},$$
and so \eqref{Propo:MassLowerBound:2}  holds.

Multiplying both sides of the above inequality with $\chi_{R_0}$ is the characteristic function of the ball $B_{\mathbb{R}^d}(O,R_0)$ in $\mathbb{R}^d$, and taking the integral with respect to $k$ on  $\mathbb{R}^d$, yield
$$\|f\chi_{R_0}\|_1\ge \int_{\mathbb{R}^d}\chi_{R_0}f(t,k)dk \geq \int_{\mathbb{R}^d}\chi_{R_0}f_0(k)e^{-(2\nu |k|^2+4|k|)t}dk$$
$$ \geq e^{-(2\nu R_0^2+4R_0)t}\int_{\mathbb{R}^d}\chi_{R_0}f_0(k)dk\ge \|f_0\chi_{R_0}\|_1,$$
and so \eqref{Propo:MassLowerBound:2a} holds true. The proof of Proposition \ref{Propo:MassLowerBound} is completed. 
\end{proof}

\subsection{Weighted $L^1_N$ $(N\geq 0)$ estimates}

For a given function $g$, let us recall the $N$-th moment of $g$ 
\begin{equation*}
\mathfrak{M}_N[g]=\int_{\mathbb{R}^d}\cE^N_kg(k)dk.
\end{equation*}

\begin{proposition}\label{Propo:MomentsPropa} Let $N\geq 0$. Suppose that $f_0(k)$ is a nonnegative  initial data
satisfying 
$$\int_{\mathbb{R}^d}f_0(k) \cE_k^N \;dk<\infty,$$
and that
   nonnegative  solutions $f(t,k)$ of \eqref{WeakTurbulenceInitial} satisfies
$$\mathfrak{M}_0[f](t)\geq \tilde{\mathfrak{M}}_0(t)=e^{-(2\nu R_0^2+4R_0)t} \|f_0(k)\chi_{R_0}\|_{L^1}>0,$$
where $\tilde{\mathfrak{M}}_0(t)$ is the quantity considered in Proposition \ref{Propo:MassLowerBound}.

Then, there exists a positive constant $C_0(\lambda_1,\lambda_2)$ is a constant depending on $\lambda_1,\lambda_2$ and independent of $N$ such that
\begin{equation}\label{Propo:MomentsPropa:1}\begin{aligned}
&\mathfrak{M}_{N+1}[Q[f]](t) - 2\nu \mathfrak{M}_N[|k|^2f](t) = \\
=&\int_{\mathbb{R}^d}Q[f](t,k)\cE_{k}^{N+1}\; dk - 2\nu  \int_{\RR^d} |k|^2f(t,k)\cE_k^N \; dk\\
\le &\ C_0(\lambda_1,\lambda_2)\left(1+\frac{e^{(4\nu R_0^2+8R_0)t} }{\|f_0(k)\chi_{R_0}\|_{L^1}^2}\right) \int_{\mathbb{R}^d}f(t,k) \cE_k^N \;dk,
\end{aligned}
\end{equation}
which implies that
   nonnegative  solutions $f(t,k)$ of \eqref{WeakTurbulenceInitial}, with $f(0,k) = f_0(k)$, satisfy
\begin{equation}\label{EE-bound}\begin{aligned}\mathfrak{M}_N[f](t)= \int_{\mathbb{R}^d}f(t,k)\cE_k^N\;dk
 \le& \ e^{\mathcal{C}(\lambda_1,\lambda_2)\left(t+\frac{e^{(4\nu R_0^2+8R_0)t} }{(4\nu R_0^2+8R_0)\|f_0(k)\chi_{R_0}\|_{L^1}^2}\right)}\int_{\mathbb{R}^d}f_0(k) \cE_k^N \;dk,\end{aligned}
\end{equation}
where $\mathcal{C}(\lambda_1,\lambda_2)$ is a constant depending on $\lambda_1,\lambda_2$.


 \end{proposition}
 \begin{remark}
 Note that \eqref{Propo:MomentsPropa:1} says that the $N$-th moment of $f$ only depends on the $N$-th moment of the initial data and the parameter $R_0$ defined in Proposition \ref{Propo:MassLowerBound}.
 \end{remark}


\begin{proof}[Proof of Proposition \ref{Propo:MomentsPropa}]
Using $\varphi = \cE_{k}^N$ as a test function in \eqref{WeakTurbulenceInitial}, we obtain 
$$\begin{aligned}
& \frac{d}{dt}\mathfrak{M}_N[f] + 2\nu \mathfrak{M}_N[|k|^2f]=\\
= & \frac{d}{dt}\int_{\mathbb{R}^d}f(t,k)\cE^N_{k}dk + 2\nu  \int_{\RR^d} |k|^2f(t,k)\cE_k^N \; dk =\int_{\mathbb{R}^d}Q[f](t,k)\cE_{k}^Ndk.\end{aligned}
$$
As a direct consequence of Lemma \ref{Propo:C12}, the following inequality holds true
\begin{equation}\label{E}
\frac{d}{dt}\int_{\mathbb{R}^d}f(t,k)\cE^N_{k}dk + 2\nu  \int_{\RR^d} |k|^2f(t,k)\cE_k^N \; dk \le \frac{C}{\mathfrak{M}_0[f]}\mathfrak{M}_{N+1}[f(t)] =  \frac{C}{\mathfrak{M}_0[f]}\int_{\RR^d} f(t,k)\cE_k^{N+1} \; dk.
\end{equation}

Notice that $$|k|^2=\frac{\omega_k^2-\lambda_1}{\lambda_2},$$
we  get the following moment equation
$$
\frac{d}{dt}\mathfrak{M}_N[f(t)] + \frac{2\nu}{\lambda_2}\mathfrak{M}_{N+2}[f(t)]-  \frac{2\nu\lambda_1}{\lambda_2}\mathfrak{M}_{N}[f(t)] \le \frac{C}{\mathfrak{M}_0[f]}\mathfrak{M}_{N+1}[f(t)].
$$

Using the fact that $$\mathfrak{M}_0[f]\geq e^{-(2\nu R_0^2+4R_0)T} \|f_0(k)\chi_{R_0}\|_{L^1},$$ we deduce from \eqref{E} 
$$\begin{aligned}
& \frac{d}{dt}\int_{\mathbb{R}^d}f(t,k)\cE^N_{k}dk + 2\nu  \int_{\RR^d} |k|^2f(t,k)\cE_k^N \; dk\\ \le & \frac{C}{\mathfrak{M}_0[f]}\mathfrak{M}_{N+1}[f(t)] \le   \frac{Ce^{(2\nu R_0^2+4R_0)T} }{\|f_0(k)\chi_{R_0}\|_{L^1}}\int_{\RR^d} f(t,k)\cE_k^{N+1} \; dk\end{aligned}.
$$

Now since $$\begin{aligned} \frac{Ce^{(2\nu R_0^2+4R_0)t} }{\|f_0(k)\chi_{R_0}\|_{L^1}}\cE_k^{N+1} - 2\nu |k|^2\cE_k^N
 = & \left(\lambda_1+\lambda_2|k|^2\right)^{\frac{N}{2}} \left(\frac{Ce^{(2\nu R_0^2+4R_0)t} }{\|f_0(k)\chi_{R_0}\|_{L^1}}\left(\lambda_1+\lambda_2|k|^2\right)^{\frac{1}{2}}- 2\nu |k|^2\right),\end{aligned}$$ 
and observing that $\frac{Ce^{(2\nu R_0^2+4R_0)t} }{ \|f_0(k)\chi_{R_0}\|_{L^1}}\left(\lambda_1+\lambda_2|k|^2\right)^{\frac{1}{2}}- 2\nu |k|^2$ is bounded uniformly by some constant $\mathcal{C}(\lambda_1,\lambda_1)\left(1+\frac{e^{(4\nu R_0^2+8R_0)t} }{\|f_0(k)\chi_{R_0}\|_{L^1}^2}\right)$, we can bound 
$$\frac{C}{\tilde{\mathfrak{M}}_0(t)}\cE_k^{N+1} - 2\nu |k|^2\cE_k^N \le \mathcal{C}(\lambda_1,\lambda_2)\left(1+\frac{e^{(4\nu R_0^2+8R_0)t} }{\|f_0(k)\chi_{R_0}\|_{L^1}^2}\right)\left(\lambda_1+\lambda_2|k|^2\right)^{\frac{N}{2}}.$$ 
The above estimate means that the difference 
$$\begin{aligned}
& \frac{Ce^{(2\nu R_0^2+4R_0)t}}{ \|f_0(k)\chi_{R_0}\|_{L^1}}\int_{\mathbb{R}^d}f(t,k)\cE_{k}^{N+1}\; dk - 2\nu  \int_{\RR^d} |k|^2f(t,k)\cE_k^N \; dk\\
=& \int_{\mathbb{R}^d}f(t,k)\left(\frac{Ce^{(2\nu R_0^2+4R_0)t} }{\|f_0(k)\chi_{R_0}\|_{L^1}}\cE_k^{N+1} - 2\nu |k|^2\cE_k^N \right)\; dk,\end{aligned}
$$
is smaller than  
$
\mathcal{C}(\lambda_1,\lambda_2)\left(1+\frac{e^{(4\nu R_0^2+8R_0)t}}{\|f_0(k)\chi_{R_0}\|_{L^1}^2}\right)\int_{\mathbb{R}^d}f(t,k)\cE_{k}^N\; dk,
$
which immediately leads to
$$
\frac{d}{dt}\int_{\mathbb{R}^d}f(t,k)\cE^N_{k}\; dk \le \mathcal{C}(\lambda_1,\lambda_2)\left({1}+\frac{e^{(4\nu R_0^2+8R_0)t}}{\|f_0(k)\chi_{R_0}\|_{L^1}^2}\right)\int_{\mathbb{R}^d}f(t,k)\cE_{k}^N\; dk.
$$
Inequality \eqref{EE-bound} then follows as a consequence of the above inequality.

\end{proof}

\section{Holder estimates for $Q[f]$}\label{Sec:HolderEstimate}

In this section, we study the H\"older continuity of the collision operator $Q[f]$ with respect to weighted $L^1_N$ norm.

\begin{proposition}\label{Propo:HolderC12} Let $M,N\ge 0$, and let $V_M$ be any bounded subset of $L^1_{N+2}(\RR^d)$, with the $L^1_{N+2}$ norms bounded from above by $M$ and the  $L^1$ norms bounded from below by $M'$. Then, there exists a constant $C_{M,M',N}$, depending on $M,M',N$, so that 
\begin{equation}\label{Propo:HolderC12:1} 
\|Q[g]-Q[h]\|_{L^1_{N}}\le \left(\frac{C}{\mathfrak{M}_0[|g|]\mathfrak{M}_0[|h|]}+\frac{C}{\mathfrak{M}_0[|g|]}\right)\| g-h\|_{L^1_N}^{\frac12} \le C_{M,M'}  \| g-h\|_{L^1_N}^{\frac12} 
\end{equation}
for all $g,h\in V_M$. 
\end{proposition}

We first prove the following lemma. 
\begin{lemma}\label{Lemma:Holder} Let $M,N>0$, and let ${V}_M$ be any bounded subset of $L^1(\RR^d) \cap L^1_{N+1}(\RR^d)$, with the $L^1_{N+2}$ norms bounded from above by $M$ and the  $L^1$ norms bounded from below by $M'$. Then, there exists a constant $C_{M,M'}$, depending on $M,M'$, so that 
\begin{equation}\label{Q-L1bound} 
\|Q[g]-Q[h]\|_{L^1_{N}} \le \left(\frac{C}{\mathfrak{M}_0[|g|]\mathfrak{M}_0[|h|]}+\frac{C}{\mathfrak{M}_0[|g|]}\right)\| g - h\|_{L^1_{N+1}}\le C_{M,M'}  \| g - h\|_{L^1_{N+1}}
\end{equation}
for all $g,h\in {V}_M$. 
\end{lemma}
\begin{proof} 
We first compute the difference between $Q[g]$ and $Q[h]$ $$Q[g] - Q[h] = 
\iint_{\mathbb{R}^{2d}} \Big[ R_{k,k_1,k_2}[g] - R_{k,k_1,k_2}[h]   - 2  ( R_{k_1,k,k_2}[g] - R_{k_1,k,k_2}[h])  \Big] dk_1dk_2, $$
whose $L^1_N$-norm is 
$$ 
\begin{aligned}
\|Q[g]-Q[h]\|_{L^1_N}  & = \int_{\mathbb{R}^d} \cE_k^N|Q[g](k)-Q[h](k)|dk
\\
&\le \iiint_{\mathbb{R}^{3d}} \cE_k^N |R_{k,k_1,k_2}[g] - R_{k,k_1,k_2}[h]| \; dkdk_1dk_2
\\&\quad 
+  2 \iiint_{\mathbb{R}^{3d}} \cE_k^N |  R_{k_1,k,k_2}[g] - R_{k_1,k,k_2}[h] | dk dk_1dk_2 
\\& = \iiint_{\mathbb{R}^{3d}}  |R_{k,k_1,k_2}[g] - R_{k,k_1,k_2}[h]| \Big( \cE_k^N + \cE_{k_1}^N + \cE_{k_2}^N\Big) \; dkdk_1dk_2.
\end{aligned}
$$
Recalling that $$\begin{aligned}
R_{k,k_1,k_2} [g]=  C|V_{k,k_1,k_2}|^2\delta(k-k_1-k_2)\mathcal{L}_g(\omega_k -\omega_{k_1}-\omega_{k_2})(g_1g_2-gg_1-gg_2),
\end{aligned}
$$
we find the following estimate on $\|Q[g]-Q[h]\|_{L^1_N}$
\begin{equation}\label{Lemma:Holder:E1} 
 \|Q[g]-Q[h]\|_{L^1_N}  \ \le \ \mathbb{J}_1 + \mathbb{J}_2,
\end{equation}
where
\begin{equation}\label{Lemma:Holder:E2} 
\begin{aligned}
\mathbb{J}_1  := & \iiint_{\mathbb{R}^{3d}}|V_{k,k_1,k_2}|^2\delta(k-k_1-k_2) \Big|\mathcal{L}_g(\omega_k -\omega_{k_1}-\omega_{k_2})g_1g_2\\
&\ \ \ \ \ \ \  \ \ \ \ \ \ \   \ \ \ \ \ \ \ \ \ \ \ \ -\mathcal{L}_h(\omega_k -\omega_{k_1}-\omega_{k_2})h_1h_2\Big|\Big( \cE_k^N + \cE_{k_1}^N + \cE_{k_2}^N\Big)dkdk_1dk_2,\\
\mathbb{J}_2  := & 2\iiint_{\mathbb{R}^{3d}}|V_{k_1,k,k_2}|^2\delta(k_1-k-k_2) \Big|\mathcal{L}_g(\omega_{k_1} -\omega_{k}-\omega_{k_2})gg_2\\
&\ \ \ \ \ \ \  \ \ \ \ \ \ \   \ \ \ \ \ \ \ \ \ \ \ \ -\mathcal{L}_h(\omega_{k_1} -\omega_{k}-\omega_{k_2})hh_2\Big|\Big( \cE_k^N + \cE_{k_1}^N + \cE_{k_2}^N\Big)dkdk_1dk_2.
\end{aligned}
\end{equation}
Let us now split the proof into two steps.
{\\\bf Step 1: Estimating $\mathbb{J}_1$.} Define the quantity  inside the triple integral of $\mathbb{J}_1$ after dropping $\Big( \cE_k^N + \cE_{k_1}^N + \cE_{k_2}^N\Big)$ to be $J_1$
\begin{equation*}
\begin{aligned}
J_1 \ := &\  |V_{k,k_1,k_2}|^2\delta(k-k_1-k_2) \Big|\mathcal{L}_g(\omega_k -\omega_{k_1}-\omega_{k_2})g_1g_2-\mathcal{L}_h(\omega_k -\omega_{k_1}-\omega_{k_2})h_1h_2\Big|,
\end{aligned}
\end{equation*}
which, by the triangle inequality, can be bounded as
\begin{equation*}
\begin{aligned}
J_1 \ \le &\  |V_{k,k_1,k_2}|^2\delta(k-k_1-k_2) \mathcal{L}_g(\omega_k -\omega_{k_1}-\omega_{k_2})|g_1g_2-h_1h_2|\\
\  &\ + |V_{k,k_1,k_2}|^2\delta(k-k_1-k_2)\Big| \mathcal{L}_g(\omega_k -\omega_{k_1}-\omega_{k_2})-\mathcal{L}_h(\omega_k -\omega_{k_1}-\omega_{k_2})\Big||h_1h_2|.
\end{aligned}
\end{equation*}
Define the two terms on the right hand side of the above inequality to be $J_{11}$ and $J_{12}$, respectively. 
\\ Let us now study $J_{11}$ in details. Using the definition of $\mathcal{L}_g$ and the triangle inequality
$$|g_1g_2-h_1h_2|\le |g_1||g_2-h_2|+|h_2||g_1-h_1|,$$
yields the following estimate on $J_{11}$
$$
\begin{aligned}
J_{11} \ \le & \ C|k||k_1||k_2|\delta(k-k_1-k_2)\frac{\Gamma_{g,k,k_1,k_2}}{(\omega_k-\omega_{k_1}-\omega_{k_2})^2+\Gamma_{g,k,k_1,k_2}^2}|g_1||g_2-h_2|\\
\ &\ + C|k||k_1||k_2|\delta(k-k_1-k_2)\frac{\Gamma_{g,k,k_1,k_2}}{(\omega_k-\omega_{k_1}-\omega_{k_2})^2+\Gamma_{g,k,k_1,k_2}^2}|h_2||g_1-h_1|.
\end{aligned}
$$
By the inequality
$$(\omega_k-\omega_{k_1}-\omega_{k_2})^2+\Gamma_{g,k,k_1,k_2}^2 \ge \Gamma_{g,k,k_1,k_2}^2,$$
we can bound $J_{11}$ as
$$
\begin{aligned}
J_{11} \ \le & \ C|k||k_1||k_2|\delta(k-k_1-k_2)\frac{1}{\Gamma_{g,k,k_1,k_2}}|g_1||g_2-h_2|\\
\ &\ + C|k||k_1||k_2|\delta(k-k_1-k_2)\frac{1}{\Gamma_{g,k,k_1,k_2}}|h_2||g_1-h_1|.
\end{aligned}
$$
The right hand side of the above inequality can be estimated by employing the following Cauchy-Schwarz inequality
$$\Gamma_{g,k,k_1,k_2} =  \mathfrak{M}_0[|g|]\left(|k|^2+|k_1|^2+|k_2|^2\right)\geq \mathfrak{M}_0[|g|]\left(|k_1|^2+|k_2|^2\right)\geq 2\mathfrak{M}_0[|g|]|k_1||k_2|,$$
where we have just used the lower bound of $\mathfrak{M}_0[|g|]$, yielding
$$
\begin{aligned}
J_{11} \ \le & \ \frac{C}{\mathfrak{M}_0[|g|]}|k|\delta(k-k_1-k_2)|g_1||g_2-h_2|
 + \frac{C}{\mathfrak{M}_0[|g|]}|k|\delta(k-k_1-k_2)|h_2||g_1-h_1|.
\end{aligned}
$$
Multiplying the above inequality with $\Big( \cE_k^N + \cE_{k_1}^N + \cE_{k_2}^N\Big)$ and integrating in $k$, $k_1$ and $k_2$ lead to
$$
\begin{aligned}
&\ \iiint_{\mathbb{R}^{3d}}J_{11}\Big( \cE_k^N + \cE_{k_1}^N + \cE_{k_2}^N\Big)dkdk_1dk_2\\
 \ \le & \ \iiint_{\mathbb{R}^{3d}}\frac{C}{\mathfrak{M}_0[|g|]}||k|\delta(k-k_1-k_2)\left[|g_1||g_2-h_2|
 + |h_2||g_1-h_1|\right]\Big( \cE_k^N + \cE_{k_1}^N + \cE_{k_2}^N\Big)dkdk_1dk_2.
\end{aligned}
$$
Using the resonant condition  $k = k_1 + k_2$, we reduce the triple integral on the right hand side to a double integral 
$$
\begin{aligned}
&\ \iiint_{\mathbb{R}^3}J_{11}\Big( \cE_k^N + \cE_{k_1}^N + \cE_{k_2}^N\Big)dkdk_1dk_2\\
 \ \le & \ \frac{C}{\mathfrak{M}_0[|g|]}\iint_{\mathbb{R}^{2d}}|k_1+k_2|\left[|g_1||g_2-h_2|
 + |h_2||g_1-h_1|\right]\Big(\cE_{k_1}^N + \cE_{k_2}^N\Big)dk_1dk_2,
\end{aligned}
$$
where, we have just used the inequality $$\cE_{k_1+k_2}^N\le C\cE_{k_1}^N + C\cE_{k_2}^N,$$
proved in Proposition \ref{Propo:C12}, to bound the sum $\cE_{k}^N + \cE_{k_1}^N + \cE_{k_2}^N$ by $C\Big(\cE_{k_1}^N + \cE_{k_2}^N\Big)$.
\\ Observing that
$$|k_1+k_2|\Big(\cE_{k_1}^N + \cE_{k_2}^N\Big)\le (|k_1|+|k_2|)\Big(\cE_{k_1}^N + \cE_{k_2}^N\Big)\le C \Big(\cE_{k_1}^{N+1} + \cE_{k_2}^{N+1}\Big),$$
we find
$$
\begin{aligned}
&\ \iiint_{\mathbb{R}^{3d}}J_{11}\Big( \cE_k^N + \cE_{k_1}^N + \cE_{k_2}^N\Big)dkdk_1dk_2\\
 \ \le & \ \frac{C}{\mathfrak{M}_0[|g|]}\iint_{\mathbb{R}^{2d}}\left[|g_1||g_2-h_2|
 + |h_2||g_1-h_1|\right]\Big(\cE_{k_1}^{N+1} + \cE_{k_2}^{N+1}\Big)dk_1dk_2,
\end{aligned}
$$
which immediately leads to
\begin{equation}\label{Lemma:Holder:E3} 
\begin{aligned}
&\ \iiint_{\mathbb{R}^{3d}}J_{11}\Big( \cE_k^N + \cE_{k_1}^N + \cE_{k_2}^N\Big)dkdk_1dk_2\\
 \ \le & \ \frac{C}{\mathfrak{M}_0[|g|]}\|g-h\|_{L^1_{N+1}}\left(\|g\|_{L^1}+\|g\|_{L^1_{N+1}}+\|h\|_{L^1}+\|h\|_{L^1_{N+1}}\right)\\
 \ \le & \ \frac{C}{\mathfrak{M}_0[|g|]}\|g-h\|_{L^1_{N+1}}\left(\|g\|_{L^1_{N+1}}+\|h\|_{L^1_{N+1}}\right).
\end{aligned}
\end{equation} 
Now, let us look at $J_{12}$, which can be written as
$$
\begin{aligned}
J_{12} \ = & \ C|k||k_1||k_2|\delta(k-k_1-k_2)|h_1h_2|\times\\
& \ \times \left|\frac{\Gamma_{g,k,k_1,k_2}[(\omega_k-\omega_{k_1}-\omega_{k_2})^2+\Gamma_{h,k,k_1,k_2}^2]-\Gamma_{h,k,k_1,k_2}[(\omega_k-\omega_{k_1}-\omega_{k_2})^2+\Gamma_{g,k,k_1,k_2}^2]}{[(\omega_k-\omega_{k_1}-\omega_{k_2})^2+\Gamma_{g,k,k_1,k_2}^2][(\omega_k-\omega_{k_1}-\omega_{k_2})^2+\Gamma_{h,k,k_1,k_2}^2]}\right|\\
\ = & \ C|k||k_1||k_2|\delta(k-k_1-k_2)|h_1h_2|\times\\
& \ \times \frac{|(\omega_k-\omega_{k_1}-\omega_{k_2})^2-\Gamma_{g,k,k_1,k_2}\Gamma_{h,k,k_1,k_2}||\Gamma_{g,k,k_1,k_2}-\Gamma_{h,k,k_1,k_2}|}{[(\omega_k-\omega_{k_1}-\omega_{k_2})^2+\Gamma_{g,k,k_1,k_2}^2][(\omega_k-\omega_{k_1}-\omega_{k_2})^2+\Gamma_{h,k,k_1,k_2}^2]}
\end{aligned}
$$
It follows from the Cauchy-Schwarz inequality that
$$
\begin{aligned}& \ [(\omega_k-\omega_{k_1}-\omega_{k_2})^2+\Gamma_{g,k,k_1,k_2}^2][(\omega_k-\omega_{k_1}-\omega_{k_2})^2+\Gamma_{h,k,k_1,k_2}^2] \\ 
\ge &\ |(\omega_k-\omega_{k_1}-\omega_{k_2})^2-\Gamma_{g,k,k_1,k_2}\Gamma_{h,k,k_1,k_2}||(\omega_k-\omega_{k_1}-\omega_{k_2})^2+\Gamma_{g,k,k_1,k_2}\Gamma_{h,k,k_1,k_2}|\\
 \ge &\ |(\omega_k-\omega_{k_1}-\omega_{k_2})^2-\Gamma_{g,k,k_1,k_2}\Gamma_{h,k,k_1,k_2}|\Gamma_{g,k,k_1,k_2}\Gamma_{h,k,k_1,k_2},
\end{aligned}
$$
from which, we obtain the following estimate on $J_{12}$
$$
\begin{aligned}
J_{12} 
\ \le & \ C|k||k_1||k_2||h_1h_2|\delta(k-k_1-k_2)\frac{|\Gamma_{g,k,k_1,k_2}-\Gamma_{h,k,k_1,k_2}|}{\Gamma_{g,k,k_1,k_2}\Gamma_{h,k,k_1,k_2}}.
\end{aligned}
$$
The numerator of the fraction on the right hand side has the following interesting property
$$
\begin{aligned}
|\Gamma_{g,k,k_1,k_2}-\Gamma_{h,k,k_1,k_2}|  \ = &\  C\left|(k^2+k_1^2+k_2^2)\mathfrak{M}_0[|g|-|h|]\right|,
\end{aligned}
$$
which can be bounded as follows
$$
\begin{aligned}
|\Gamma_{g,k,k_1,k_2}-\Gamma_{h,k,k_1,k_2}|  \ \le &\  C(k^2+k_1^2+k_2^2)\|g-h\|_{L^1},
\end{aligned}
$$
yielding an upper bound on $J_{12}$
$$
\begin{aligned}
J_{12} 
\ \le & \ C|k||k_1||k_2||h_1h_2|\delta(k-k_1-k_2)\frac{\|g-h\|_{L^1},}{(k^2+k_1^2+k_2^2) \mathfrak{M}_0[|g|]\mathfrak{M}_0[|h|]}.
\end{aligned}
$$
By the Cauchy-Schwarz inequality
$$k^2+k_1^2+k_2^2\ge k_1^2+k_2^2 \ge 2|k_1||k_2|,$$
and the lower bound on $\mathfrak{M}_0[|g|]$ and $\mathfrak{M}_0[|h|]$, the following estimate on $J_{12}$ then follows
$$
\begin{aligned}
J_{12} 
\ \le & \ \frac{C}{\mathfrak{M}_0[|g|]\mathfrak{M}_0[|h|]}||k||h_1h_2|\delta(k-k_1-k_2)\|g-h\|_{L^1}.
\end{aligned}
$$
Multiplying the above inequality with $\Big( \cE_k^N + \cE_{k_1}^N + \cE_{k_2}^N\Big)$ and integrate in $k$, $k_1$ and $k_2$, the same argument used to deduce \eqref{Lemma:Holder:E3}  leads to
\begin{equation}\label{Lemma:Holder:E4}
\begin{aligned}
\iiint_{\mathbb{R}^{3d}}J_{12} \Big( \cE_k^N + \cE_{k_1}^N + \cE_{k_2}^N\Big)dkdk_1dk_2
 \le & \frac{C}{\mathfrak{M}_0[|g|]\mathfrak{M}_0[|h|]}\|g-h\|_{L^1_{N+1}}.
\end{aligned}
\end{equation}
Note that $C$ is a constant depending on $\left(\|g\|_{L^1_{N+1}}+\|h\|_{L^1_{N+1}}\right)$.
Combining \eqref{Lemma:Holder:E3} and \eqref{Lemma:Holder:E4} yields
\begin{equation}\label{Lemma:Holder:E5}
\begin{aligned}
\mathbb{J}_1\  \le &\ \left(\frac{C}{\mathfrak{M}_0[|g|]\mathfrak{M}_0[|h|]}+\frac{C}{\mathfrak{M}_0[|g|]}\right)\|g-h\|_{L^1_{N+1}},
\end{aligned}
\end{equation}
where $C$ is a constant depending on $\left(\|g\|_{L^1_{N+1}}+\|h\|_{L^1_{N+1}}\right)$.
{\\\bf Step 2: Estimating $\mathbb{J}_2$.} The proof of estimating $\mathbb{J}_2$ follows exactly the same argument used in Step 1. As a consequence, we omit some details and give only the main estimates in the sequel. First, define the quantity  inside the triple integral of $\mathbb{J}_2$ after dropping $\Big( \cE_k^N + \cE_{k_1}^N + \cE_{k_2}^N\Big)$ to be $J_2$
\begin{equation*}
\begin{aligned}
J_2 \ := &\  |V_{k_1,k,k_2}|^2\delta(k_1-k-k_2) \Big|\mathcal{L}_g(\omega_{k_1} -\omega_{k}-\omega_{k_2})gg_2-\mathcal{L}_h(\omega_{k_1} -\omega_{k}-\omega_{k_2})hh_2\Big|,
\end{aligned}
\end{equation*}
which, by the triangle inequality, can be bounded as
\begin{equation*}
\begin{aligned}
J_2 \ \le &\  |V_{k_1,k,k_2}|^2\delta(k_1-k-k_2) \mathcal{L}_g(\omega_{k_1} -\omega_{k}-\omega_{k_2})|gg_2-hh_2|\\
\  &\ + |V_{k_1,k,k_2}|^2\delta(k_1-k-k_2)\Big| \mathcal{L}_g(\omega_{k_1} -\omega_{k}-\omega_{k_2})-\mathcal{L}_h(\omega_{k_1} -\omega_{k}-\omega_{k_2})\Big||hh_2|.
\end{aligned}
\end{equation*}
We set the two terms on the right hand side of the above inequality to be $J_{21}$ and $J_{22}$, respectively. 
\\ The following estimate on $J_{21}$ is a direct consequence of the triangle inequality
$$
\begin{aligned}
J_{21} \ \le & \ |k||k_1||k_2|\delta(k_1-k-k_2)\frac{\Gamma_{g,k,k_1,k_2}}{(\omega_{k_1}-\omega_{k}-\omega_{k_2})^2+\Gamma_{g,k,k_1,k_2}^2}|g||g_2-h_2|\\
\ &\ + C|k||k_1||k_2|\delta(k_1-k-k_2)\frac{\Gamma_{g,k,k_1,k_2}}{(\omega_{k_1}-\omega_{k}-\omega_{k_2})^2+\Gamma_{g,k,k_1,k_2}^2}|h_2||g-h|.
\end{aligned}
$$
The same argument used in Step 1 can be employed, implying the following estimate on $J_{21}$
$$ \begin{aligned}
J_{21} \ \le & \ \frac{C}{\mathfrak{M}_0[|g|]}|k_1|\delta(k-k_1-k_2)|g||g_2-h_2|
 + \frac{C}{\mathfrak{M}_0[|g|]}|k_1|\delta(k-k_1-k_2)|h_2||g-h|.
\end{aligned}
$$
Multiplying the above inequality with $\Big( \cE_k^N + \cE_{k_1}^N + \cE_{k_2}^N\Big)$ and integrate in $k$, $k_1$ and $k_2$ yields
\begin{equation}\label{Lemma:Holder:E6} 
\begin{aligned}
&\ C\iiint_{\mathbb{R}^{3d}}J_{21}\Big( \cE_k^N + \cE_{k_1}^N + \cE_{k_2}^N\Big)dkdk_1dk_2\\
 \ \le & \ C\left(\|g-h\|_{L^1}+\|g-h\|_{L^1_{N+1}}\right),
\end{aligned}
\end{equation} 
where $C$ is a constant depending on $\left(\|g\|_{L^1}+\|g\|_{L^1_{N+1}}+\|h\|_{L^1}+\|h\|_{L^1_{N+1}}\right)$.
\\
Now, similar to $J_{12}$, $J_{22}$  can be bounded as
$$
\begin{aligned}
J_{22} 
\ \le & \ C|k||k_1||k_2||hh_2|\delta(k_1-k-k_2)\frac{|\Gamma_{g,k,k_1,k_2}-\Gamma_{h,k,k_1,k_2}|}{\Gamma_{g,k,k_1,k_2}\Gamma_{h,k,k_1,k_2}}.
\end{aligned}
$$
The same argument used in Step 1 can be applied and the following estimate on $J_{22}$ then follows
$$
\begin{aligned}
J_{22} 
\ \le & \ |k||hh_2|\delta(k-k_1-k_2)\|g-h\|_{L^1}.
\end{aligned}
$$
Multiplying the above inequality with $\Big( \cE_k^N + \cE_{k_1}^N + \cE_{k_2}^N\Big)$ and integrate in $k$, $k_1$ and $k_2$, we obtain
\begin{equation}\label{Lemma:Holder:E7}
\begin{aligned}
&\ \iiint_{\mathbb{R}^{3d}}J_{22} \Big( \cE_k^N + \cE_{k_1}^N + \cE_{k_2}^N\Big)dkdk_1dk_2
 \le  \frac{C}{\mathfrak{M}_0[|g|]\mathfrak{M}_0[|h|]}\left(\|g-h\|_{L^1}+\|g-h\|_{L^1_{N+1}}\right),
\end{aligned}
\end{equation}
where $C$ is a constant depending on $\left(\|g\|_{L^1_{N+1}}+\|h\|_{L^1_{N+1}}\right)$.
\\
Combining \eqref{Lemma:Holder:E6} and \eqref{Lemma:Holder:E7} yields
\begin{equation}\label{Lemma:Holder:E8}
\begin{aligned}
\mathbb{J}_2\  \le &\ \left(\frac{C}{\mathfrak{M}_0[|g|]\mathfrak{M}_0[|h|]}+\frac{C}{\mathfrak{M}_0[|g|]}\right)\left(\|g-h\|_{L^1}+\|g-h\|_{L^1_{N+1}}\right)\\ \le &\ \left(\frac{C}{\mathfrak{M}_0[|g|]\mathfrak{M}_0[|h|]}+\frac{C}{\mathfrak{M}_0[|g|]}\right)\|g-h\|_{L^1_{N+1}}.
\end{aligned}
\end{equation}
Putting the two estimates \eqref{Lemma:Holder:E5} and \eqref{Lemma:Holder:E8} together with \eqref{Lemma:Holder:E1} and \eqref{Lemma:Holder:E2}, the conclusion of the Lemma then follows. 
\end{proof}
\begin{proof}[Proof of Proposition \ref{Propo:HolderC12}] The proposition now follows straightforwardly from the previous lemma. Indeed, we recall the interpolation inequality (see Lemma \ref{lem-Holder}):
$$ \| g\|_{L^1_n} \le \| g \|_{L^1_p}^{\frac{q-n}{q-p}}  \| g \|_{L^1_q}^{\frac{n-p}{q-p}}  $$
for $q>n>p$. Together with the boundedness of $g,h$ in $L^1_1 \cap L^1_{N+2}$, we obtain 
$$
\begin{aligned}
 \| g-h\|_{L^1_{N+1}}  &\le \| g-h\|_{L^1_N}^{\frac12} \| g-h\|_{L^1_{N+2}}^{\frac12} \le C_M \| g-h\|_{L^1_N}^{\frac12} 
\end{aligned}$$
Lemma \ref{Lemma:Holder} yields 
$$\|Q[g]-Q[h]\|_{L^1_N}  \le  C_{M,M',N}\| g-h\|_{L^1_N}^{\frac12} $$
which holds for all $N\geq 0 $. The proposition follows. 
\end{proof}

\section{Proof of Theorem \ref{Theorem:Main}}\label{Sec:Main}
We shall apply Theorem \ref{Theorem:ODE} for \eqref{WeakTurbulenceInitial}, which reads
$$ \partial_t f = \widetilde{Q}[f], \qquad\qquad  \widetilde{Q}[f] := Q[f] - 2\nu |k|^2f.$$

Fix an $N>1$. We choose the Banach spaces $E= L^{1}_{N}\big(\mathbb{R}^{d}\big)$,  $F= L^{1}_{N+3}\big(\mathbb{R}^{d}\big)$,  endowed with the  norms 
$$ \| f \|_{E}: =  \|f\|_{L^1_N},  \ \ \ \ \ \ \ \ \| f \|_{*}: =  \|f\|_{L^1_{N+3}}.$$  
We also define $$|f|_*:=\mathfrak{M}_{N+3}[f],$$
then
$$|f|_* \le \|f\|_*, \ \ \forall f\in F,\ \ \ |f+g|_* \le |f|_*+|g|_*, \ \ \forall f,g\in F,$$
$$\lambda|f|_* = |\lambda f|_*, \ \ \forall f\in F,\lambda\in\mathbb{R}_+,$$
and
$$|f|_* = \|f\|_{L^1_{N+3}}, \ \ \forall f\in \mathcal{S}_T.$$
Moreover, condition \eqref{LesbegueDominated} is automatically satisfied due to the Lebesgue dominated convergence theorem and  Theorem 1.2.7 \cite{BadialeSerra:SEE:2011}.

Clearly, $\mathcal{S}_T$ is a bounded and closed set with respect to the norm $\|\cdot \|_*$.By Proposition \ref{Propo:MomentsPropa}, for $f_0 \in \mathcal{S}_0\subset \mathcal{S}_T$, solutions to \eqref{WeakTurbulenceInitial} will remain in $\mathcal{S}_T$. Thus, it suffices to verify the three conditions $(\mathfrak{A})$, $(\mathfrak{B})$, $(\mathfrak{C})$  of Theorem \ref{Theorem:ODE}, then Theorem \ref{Theorem:Main} is a consequence of Theorem \ref{Theorem:ODE}. Notice that continuity condition  $(\mathfrak{A})$ follows directly from Proposition \ref{Propo:HolderC12}, we therefore only need to verify $(\mathfrak{B})$ and $(\mathfrak{C})$.


\subsection{Condition $(\mathfrak{B})$: Subtangent condition.}\label{Subtangent}
Let $f$ be an arbitrary element of the set $\mathcal{S}_T$. It suffices to prove the following claim: for all $\epsilon>0$, there exists $h_*$ depending on $f$ and $\epsilon$ such that 
\begin{equation}\label{claim}B(f+h\widetilde Q[f],h\epsilon)\cap\mathcal{S}_T\not =\emptyset , \qquad 0<h<h_*.\end{equation}
For $R>0$, let $\chi_R(k)$  be the characteristic function of the ball $B(0,R)$, and set  
\begin{equation}\label{def-wR} w_R:=f +h\widetilde Q[f_R] , \qquad\quad f_R(k)=\chi_R(k)f(k),\end{equation}
recalling $\widetilde{Q}[g] = Q[g] - 2\nu |k|^2 g$. We shall prove that for all $R>0$, there exists an $h_R$ so that $w_R$ belongs to $\mathcal{S}_T$, for all $0<h\le h_R$. It is clear that $w_R \in L^{1}(\RR^d) \cap L^1_{N+3}(\RR^d)$.

We now check the conditions {\bf S1}, {\bf S2} and {\bf S3} in \eqref{subsetS_T}.

~\\
{\bf Condition ({\bf S1}): Positivity of the set $\mathcal{S}_T$.} Note that one can write $Q[f] = Q_\mathrm{gain}[f] - Q_\mathrm{loss}[f]$, with $Q_\mathrm{gain}[f] \ge 0$ and $Q_\mathrm{loss}[f] = f Q_-[f]$. 
Since $f_R$ is compactly supported, it is clear that $\chi_RQ_-[f_R]$ is bounded by a universal positive constant $4R$,  computed in Proposition \ref{Propo:MassLowerBound}.    Hence, 
$$\begin{aligned}
 w_R 
 &= f  + h \Big( Q[f_R]  - 2 \nu |k|^2f_R \Big) 
 \\&\ge   f  -  h f_R \Big(4R +  2\nu R^2\Big) 
 \end{aligned}$$
which is nonnegative, for sufficiently small $h$; precisely, $h < \frac{h_R}{2}: =\frac{1}{2(4R + 2\nu R^2)}$. 

Suppose that $R>R_0$ are chosen large enough such that $$\|\chi_{R}u_0\|_{*} > \|\chi_{R_0}u_0\|_{*} >R^*.$$ 
Let us check \eqref{Coercivity} for $R_0<R$. By Proposition \ref{Propo:MassLowerBound} 

\begin{equation}\label{K2R0}
\chi_{R_0}\frac{w_R-f}{h}=\chi_{R_0}\tilde{Q}[f_R]\geq -(4R_0 +\nu R_0^2) f_{R_0}.
\end{equation}

Moreover
$$|w_R-f|_*= h|Q[f_R]  - 2 \nu |k|^2f_R|_*\le C_0\|f_R\|_{*} ,$$
where the last inequality follows from Proposition \ref{Propo:MomentsPropa}. That leads to
\begin{equation}\label{K1R0}
|w_R-f|_* \le \frac{C(\lambda_1,\lambda_2)e^{(2\nu R_0^2+4R_0)T}}{\|f_0(k)\chi_{R_0}\|_{L^1}}\|f\|_{*} .\end{equation}
 with $\frac{C_0(\lambda_1,\lambda_2)e^{(2\nu R_0^2+4R_0)T}}{\|f_0(k)\chi_{R_0}\|_{L^1}}$ computed in  Proposition \ref{Propo:MomentsPropa}.

~\\
{\bf Condition ({\bf S2}): Upper bound of the set $\mathcal{S}_T$.} Since $$\|f\|_{*}<(2R_*+1)e^{C_*T},$$
and 
$$\lim_{h\to 0}\|f-w_R\|_{*} =0,$$
we can choose $h_*$ small enough such that  for $0<h<h_*$
$$\|w_R\|_{*} <(2R_*+1)e^{C_*T}.$$

~\\
{\bf Condition ({\bf S3}): Lower bound of the set $\mathcal{S}_T$.} Since $$\|f\|_{*} >R^*e^{-C^*T}/2,$$
and 
$$\lim_{h\to 0}\|f-w_R\|_{*} =0,$$
we can choose $h_*$ small enough such that 
$$\|w_R\|_{*} >R^*e^{-C^*T}/2.$$

This proves the claim \eqref{claim}, and hence condition $(\mathfrak{A})$ is verified.

\subsection{Condition $(\mathfrak{C})$: One side Lipschitz condition.}\label{Lipschitz}
By the Lebesgue's dominated convergence theorem, we have that
$$
\begin{aligned}
\Big[\varphi,\phi\Big] 
&= \lim_{h\rightarrow 0^{-}}h^{-1}\big(\| \phi + h\varphi \|_E - \| \phi \|_E \big)
\\
 &= \lim_{h\rightarrow 0^{-}}h^{-1} \int_{\RR^d} ( | \phi + h \varphi| - |\phi| )  (\cE_k + \cE_k^N)\; dk
\\&
\le \int_{\mathbb{R}^d}\varphi(k)\mathrm{sign}(\phi(k)) (\cE_k + \cE_k^N)dk.
\end{aligned}$$
Hence, recalling $\widetilde Q[f] = Q[f] - 2\nu |k|^2f$, we estimate 
$$
\begin{aligned}
\big[ \widetilde{Q}[f] - \widetilde{Q}[g], f - g \big] 
& \le 
\int_{\mathbb{R}^d}[\widetilde Q[f](k)-\widetilde Q[g](k)]\mathrm{sign}((f-g)(k)) \cE_k^Ndk
\\
& \le 
\|Q[f]-Q[g]\|_E  -2 \nu   \| |k|^2(f-g)\|_E.
\end{aligned}$$
Using Lemma \ref{Lemma:Holder} and recalling $\|\cdot \|_E = \|\cdot \|_{L^1_N}$, we have
$$
\begin{aligned}
\|Q[f]-Q[g]\|_{E} 
& \le  C_{N}  \| f-g\|_{L^1_{N}}.
\end{aligned}$$
Since $C|k|^N- 2\nu |k|^{N+2}$ is always bounded by $C'|k|^N$ for $C'>0$,
we obtain 
$$
\begin{aligned}
\big[ \widetilde{Q}[f] - \widetilde{Q}[g], f - g \big]  
&\le  
  C_N \| f-g\|_{E} .\end{aligned}$$

The condition $(\mathfrak{C}$) follows. The proof of Theorem \ref{Theorem:Main} is complete.

\section{Proof of Theorem \ref{Theorem:ODE}}\label{Appendix}
The proof is divided into four parts. 
{\\\\\bf Part 1:}
According to our assumption, $\mathcal{S}_T$ is bounded by a constant $C_S$ in the norm $\|\cdot\|$,  due to the H\"older continuity property of $\mathcal{Q}[u]$, 
$$\|\mathcal{Q}[u]\|\le C_\mathcal{Q}, \ \forall u \in \mathcal{S}_T.$$ 
By our assumption, for an element $u$ in $\mathcal{S}_0\subset \mathcal{S}_T$, there exists $\xi_u>0$ such that for $0<\xi<\xi_u$, $$B(u+\xi \mathcal{Q}[u],\delta)\cap {\mathcal{S}}_T\backslash\{u+\xi \mathcal{Q}[u]\}\ne {\O},$$ for $\delta$ small enough.
\\ For a fixed $u$ and $\epsilon>0$, there exists $\xi>0$ such that $\|u-v\|\le (C_\mathcal{Q}+1)\xi$ then  $\|Q(u)-Q(v)\|\leq \frac{\epsilon}{2}$. Let $z$ be in $B\left(u+\xi Q[u],\frac{\epsilon \xi}{2}\right)\cap {\mathcal{S}_T}\backslash\{u+\xi \mathcal{Q}[u]\} $ satisfying
$$\left|\frac{z-u}{\xi}\right|_* \le  \frac{C_*}{2}\|u\|_*,\ \ \ \  \chi_{R_0}\frac{z-u}{\xi}\geq -\chi_{R_0}\frac{C^*}{2}u,$$
and define
$$t\mapsto \Theta(t)=u+\frac{t(z-u)}{\xi},~~~~t\in[0,\xi].$$

Now, we also have the following lower bound on $\Theta$
\begin{equation}\label{Theorem:ODE:E4a}
\begin{aligned}
\chi_{R_0}\Theta(t)\ = & \ \chi_{R_0}\left(u+\frac{t(z-u)}{\xi}\right) \\
\ge & \ \chi_{R_0}\left(1-\frac{tC^*}{2}\right)u\\
\ge & \ \chi_{R_0}e^{-tC^*}\Theta(0),
\end{aligned}
\end{equation}
for $\xi$ and $0\le t\le\xi \le \frac{\log2}{C^*}$.

Hence 
\begin{equation}\label{Theorem:ODE:E4aa}
\|\chi_{R_0}\Theta(t)\|_* >\frac{R^*e^{-C^*t}}{2}.
\end{equation}

We also have that
$$\begin{aligned}
\|\Theta(t)\|_* =|\Theta(t)|_* = \left|u+\frac{t(z-u)}{\xi}\right|_* & \le  |u|_*+\left|\frac{t(z-u)}{\xi}\right|_* \le  |u|_* + |u|_*\frac{tC_*}{2} \\
&= \|\Theta(0)\|_*\left(1+\frac{tC_*}{2}\right).
\end{aligned}$$
We then obtain
\begin{equation}\label{Theorem:ODE:E4}
\|\Theta(t)\|_*\le (\|\Theta(0)\|_*+1)e^{C_*t}-1 <
 (2R_*+1)e^{C_*t}.\end{equation}

Therefore, $\Theta$ maps $[0,\xi]$ into $\mathcal{S}_T$. 
It is straightforward that $$\|\Theta(t)-u\|\le \left\|\frac{t(z-u)}{{\xi}}\right\|\leq \xi\|\mathcal{Q}[u]\|+\frac{\epsilon \xi}{2}<(C_\mathcal{Q}+1)\xi,$$
which implies $$\|\mathcal{Q}[{\Theta}(t)]-\mathcal{Q}[u]\|\leq \frac{\epsilon}{2},~~\forall t\in[0,\xi].$$
Combining the above inequality and the fact that $$\|\dot{\Theta}(t)-\mathcal{Q}[u]\|=\left\|\frac{z-u}{\xi}-\mathcal{Q}[u]\right\|\leq \frac{\epsilon}{2},$$
we obtain 
\begin{equation}\label{Theorem:ODE:E1}
\|\dot{\Theta}(t)-\mathcal{Q}[{\Theta}(t)]\|\leq \epsilon,~~\forall t\in[0,\xi].
\end{equation}
 {\\\bf Part 2:} Let $\Theta$ be a solution to \eqref{Theorem:ODE:E1} on $[0,\xi]$  constructed in Part 1. 
Using the procedure of Part 1, we assume that $\Theta$ can be extended to the interval $[\tau,\tau+\tau']$.
 \\ The same arguments that lead to \eqref{Theorem:ODE:E4} imply
  $$\|\Theta(\tau+t)\|_*\le \left( (\|\Theta(\tau)\|_*+1)e^{C_*t}-1\right), \ \  \ t\in[0,\tau'].$$
Combining the above inequality with \eqref{Theorem:ODE:E4} yields
\begin{equation}\label{Theorem:ODE:E5}
\begin{aligned}
\|\Theta(\tau+t)\|_*\
\le&~~\left(\left(\|\Theta(0)\|_*+1\right)e^{C_*\tau}-1+1\right)e^{C_*t}-1\\
\le&~~\left(\|\Theta(0)\|_*+1\right)e^{C_*(\tau+t)}-1\\
<&~~ (2R_*+1)e^{C_*(\tau+t)},
\end{aligned}
\end{equation}
where the last inequality follows from the fact that $R_*\ge 1$.
\\ Similar, we also have
\begin{equation}\label{Theorem:ODE:E5a}
\begin{aligned}
\chi_{R_0}\Theta(\tau+t)\ 
\ge & \ \chi_{R_0}e^{-(\tau+t)C^*}\Theta(0),
\end{aligned}
\end{equation}
which implies
\begin{equation}\label{Theorem:ODE:E5aa}
\|\chi_{R_0}\Theta(\tau+t)\|_* >\frac{R^*e^{-C^*(\tau+t)}}{2}.
\end{equation}
{\\\\\bf Part 3:} From Part 1, there exists a solution $\Theta$ to the equation \eqref{Theorem:ODE:E1} on an interval $[0,\xi]$. Now, we have the following procedure.
\begin{itemize}
\item {\it Step 1:}  Suppose that  we can construct the solution $\Theta$ of \eqref{Theorem:ODE:E1} on $[0,\tau]$ $(\tau<T)$, where $\Theta(0)\in\mathcal{S}_0\cap B_*\Big(O,R_*\Big)\backslash \overline{B_*\Big(O,R^*\Big)}$. Since due to Part 2 $\Theta(\tau)\in\mathcal{S}_\tau$, by the same process as in Part 1 and by \eqref{Theorem:ODE:E4}, \eqref{Theorem:ODE:E4a} \eqref{Theorem:ODE:E4aa}, \eqref{Theorem:ODE:E5}, \eqref{Theorem:ODE:E5a} and \eqref{Theorem:ODE:E5aa} the solution $\Theta$ could be extended to  $[\tau,\tau+h_\tau]$ where $\tau+h_\tau\le T$. 
\item {\it Step 2:} Suppose that we can construct the solution $\Theta$ of \eqref{Theorem:ODE:E1}  on a series of intervals $[0,\tau_1]$, $[\tau_1,\tau_2]$, $\cdots$, $[\tau_n,\tau_{n+1}]$, $\cdots$. Since the increasing sequence $\{\tau_n\}$ is bounded by $T$, it has a limit, noted by $\tau.$
 Moreover
\begin{equation}\label{Theorem:ODE:E5b}
\begin{aligned}
\|\Theta(t)\|_*\le & \ (\|\Theta(0)\|_*+1)e^{C_*t}-1 <  \
 (2R_*+1)e^{C_*t},  &  \forall t\in[0,\tau),\\
\chi_{R_0}\Theta(t)\ 
\ge & \ \chi_{R_0}e^{-tC^*}\Theta(0), & \forall t\in[0,\tau),
\end{aligned}
\end{equation}
and
\begin{equation}\label{Theorem:ODE:E5bb}
\|\chi_{R_0}\Theta(t)\|_* >\frac{R^*e^{-C^*t}}{2}, \forall t\in[0,\tau).
\end{equation} 
Recall that $\|\mathcal{Q}({\Theta})\|$ is bounded by $C_\mathcal{Q}$ on $[\tau_n,\tau_{n+1}]$ for all $n\in\mathbb{N},$ then $\|\dot{\Theta}\|$ is bounded by $\epsilon+C_\mathcal{Q}$ on $[0,\tau)$. As a consequence,  $\Theta(\tau)$ can be defined to be the limit of $\Theta(\tau_n)$ with respect to the norm $\|\cdot\|$. 
That, together with \eqref{LesbegueDominated} and the fact that $\mathcal{S}_\tau$ is closed with respect to $\|\cdot\|_*$, implies that $\Theta$ is a solution of \eqref{Theorem:ODE:E1} on $[0,\tau]$. In addition \eqref{Theorem:ODE:E5b} and \eqref{Theorem:ODE:E5bb} also hold true on $[0,\tau]$. 
\end{itemize}
As a consequence, if
 the solution $\Theta$ can be defined on $[0,T_0)$, $T_0<T$, it could be extended to $[0,T_0]$. Now, we suppose that $[0,T_0]$ is the maximal  closed interval that $\Theta$ could be defined, by Step 1 and Step 2. $\Theta$ could be extended to a larger interval $[T_0,T_0+T_h]$, which means that $T=T_0$ and $\Theta$ is defined on the whole interval $[0,T]$.
{\\\\\bf Part 4:} Finally, let us consider a sequence of solution $\{u^\epsilon\}$ to \eqref{Theorem:ODE:E1} on $[0,T]$. We will prove that this is a Cauchy sequence. 
Let $\{u^\epsilon\}$ and $\{v^\epsilon\}$ be two sequences of solutions to \eqref{Theorem:ODE:E1} on $[0,T]$. We note that $u^\epsilon$ and $v^\epsilon$ are affine functions on $[0,T]$. Moreover by the one-side Lipschitz condition
\begin{eqnarray*}
\frac{d}{dt} \|u^\epsilon(t)-v^\epsilon(t)\|&=&\Big[u^\epsilon(t)-v^\epsilon(t),\dot{u}^\epsilon(t)-\dot{v}^\epsilon(t)\Big]\\
&\le& \Big[u^\epsilon(t)-v^\epsilon(t),\mathcal{Q}[u^\epsilon(t)]-\mathcal{Q}[v^\epsilon(t)]\Big]+2\epsilon\\
&\le& C\|u^\epsilon(t)-v^\epsilon(t)\|+2\epsilon, 
\end{eqnarray*}
for a.e. $t\in[0,T]$, which leads to
$$\|u^\epsilon(t)-v^\epsilon(t)\|\le 2\epsilon\frac{e^{LT}}{L}.$$
By letting $\epsilon$ tend to $0$, $u^\epsilon\to u$ uniformly on $[0,T]$. It is straightforward that $u$ is a solution to \eqref{Theorem_ODE_Eq}.

~\\
{\bf Acknowledgements:}  
This work has been partially supported by NSF grants DMS 143064 and RNMS (Ki-Net) DMS-1107444, DMS (Ki-Net) 1107291. M.-B Tran is partially supported by NSF Grants DMS-1814149 and DMS-1854453.

\bibliographystyle{plain}
\bibliography{QuantumBoltzmann}

\end{document}